\documentclass{article}
\usepackage{amsmath, amssymb, mathdots, tikz,amsthm, epsfig, graphicx}
\newtheorem{theorem}{Theorem}
\newtheorem{lemma}{Lemma}
\newtheorem{definition}{Definition}

\newtheorem{corollary}{Corollary}

\title{Hypergraphs with Spectral Radius at most $(r-1)!\sqrt[r]{2+\sqrt{5}}$}
\author{
Linyuan Lu
\thanks{Department of Mathematics, University of South Carolina, Columbia, SC 29208,
({\tt lu@math.sc.edu}). This author was supported in part by NSF
grant DMS 1300547 and ONR grant N00014-13-1-0717. }
\and
Shoudong Man
\thanks{Department of Mathematics, Renmin University of China, Beijing 100872, P.R. China,({\tt shoudongmanbj@ruc.edu.cn}; (+86)13718204350).
This author was supported by the grant of China Scholarship Council (CSC), Grant No. 201306360100. }
}
\date{October 08, 2014}
\begin{document}
\maketitle
\begin{abstract}
  In our previous paper, we classified all $r$-uniform hypergraphs with
spectral radius at most $(r-1)!\sqrt[r]{4}$, which directly
generalizes Smith's theorem for the graph case $r=2$. It is nature to
ask the structures of the hypergraphs with spectral radius slightly beyond
$(r-1)!\sqrt[r]{4}$. For $r=2$, the graphs with spectral radius
at most $\sqrt{2+\sqrt{5}}$ are classified by [{\em Brouwer-Neumaier,  Linear Algebra
  Appl., 1989}].  Here we consider the $r$-uniform hypergraphs $H$ with spectral
radius at most $(r-1)!\sqrt[r]{2+\sqrt{5}}$.
We show that $H$ must have a quipus-structure, which is similar
to the graphs with spectral radius at most $\frac{3}{2}\sqrt{2}$
[{\em Woo-Neumaier,  Graphs Combin., 2007}].
\end{abstract}

\section{Introduction}
The spectral radius $\rho(G)$ of a graph $G$ is the largest eigenvalue of its
adjacency matrix.  The (simple undirected connected) graphs with small
spectral radius have been well-studied in the literature.
In 1970 Smith classified all connected graphs with spectral radius
at most $2$. The graphs $G$ with $\rho(G)<2$ are
 simple Dynkin Diagrams  $A_n$, $D_n$, $E_6$, $E_7$, and $E_8$,
while the graphs $G$ with $\rho(G)=2$
 simply extend Dynkin
 Diagrams $\tilde A_n$, $\tilde D_n$, $\tilde E_6$, $\tilde E_7$, and $\tilde E_
8$.
 Cvetkovi\'c et al.~\cite{CDG} gave a nearly complete description
of all graphs $G$ with $2 < \rho(G) < \sqrt{2 + \sqrt{5}}$. Their description
was completed by Brouwer and Neumaier \cite{BN}.
Namely,
$E(1,b,c)$ for $b=2, c\geq 6$ or $b\geq 3, c\geq 4$,
$E(2,2,c)$ for $c\geq 3$, and $G_{1,a:b:1,c}$ for $a\geq 3$, $c\geq 2$, $b>a+c$.
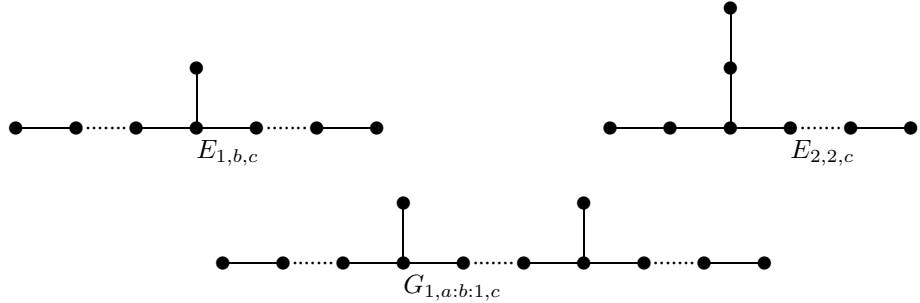
\begin{center}
\setlength{\unitlength}{0.8cm}
\begin{picture}(8,2.0)
\multiput(0,0)(1,0){7}{\circle*{0.2}}
\multiput(0,0)(1,0){1}{\line(1,0){1}}
\multiput(1.2,0)(0.1,0){7}{\circle*{0.05}}
\multiput(2,0)(1,0){2}{\line(1,0){1}}
\multiput(4.2,0)(0.1,0){7}{\circle*{0.05}}
\multiput(5,0)(1,0){1}{\line(1,0){1}}
\put(3,1){\circle*{0.2}}
\put(3,0){\line(0,1){1}}
\put(3,-0.5){$E_{1,b,c}$}
\end{picture}
\hfil
\begin{picture}(6,2.0)
\multiput(0,0)(1,0){6}{\circle*{0.2}}
\multiput(0,0)(1,0){3}{\line(1,0){1}}
\multiput(3.2,0)(0.1,0){7}{\circle*{0.05}}
\multiput(4,0)(1,0){1}{\line(1,0){1}}
\put(2,1){\circle*{0.2}}
\put(2,2){\circle*{0.2}}
\put(2,0){\line(0,1){1}}
\put(2,1){\line(0,1){1}}
\put(3,-0.5){$E_{2,2,c}$}
\end{picture}\\

\begin{picture}(9,2.2)
\multiput(0,0)(1,0){10}{\circle*{0.2}}
\multiput(0,0)(1,0){1}{\line(1,0){1}}
\multiput(1.2,0)(0.1,0){7}{\circle*{0.05}}
\multiput(2,0)(1,0){2}{\line(1,0){1}}
\multiput(4.2,0)(0.1,0){7}{\circle*{0.05}}
\multiput(5,0)(1,0){2}{\line(1,0){1}}
\multiput(7.2,0)(0.1,0){7}{\circle*{0.05}}
\multiput(8,0)(1,0){1}{\line(1,0){1}}
\put(3,1){\circle*{0.2}}
\put(3,0){\line(0,1){1}}
\put(6,1){\circle*{0.2}}
\put(6,0){\line(0,1){1}}
\put(3,-0.5){$G_{1,a:b:1,c}$}
\end{picture}
\end{center}
\vspace*{8mm}

Wang et al. \cite{wang} studied some graphs with spectral radius close to $\frac
{3}{2}{\sqrt{2}}$.
Woo and Neumaier \cite{WN} proved that any connected graph $G$ with $\sqrt{
2 + \sqrt{5}}<\rho(G)< \frac{3}{2}{\sqrt{2}}$ is one of the following graphs.
\begin{enumerate}
\item If $G$ has maximum degree at least $4$, then $G$ is a {\it
    dagger} (i.e., a tree obtained by attaching a path to a leaf vertex of the star $
S_5$).

\item If $G$ is a tree with maximum degree at most $3$, then $G$ is an {\it open
 quipu} (i.e., all the vertices of degree $3$ lie on a path).

\item If $G$ contains a cycle, then $G$ is a {\it closed quipu}
(i.e., a unicyclic graph with maximum degree at most $3$ satisfies that all the
 vertices of degree $3$ lie on a cycle).
\end{enumerate}

Lan-Lu \cite{specdiam2} proved that for
any open quipu $G$ on $n$ vertices ($n\geq 6$)
 with spectral radius less than $\frac{3}{2}{\sqrt{2}}$, its diameter
  $D(G)$ satisfies $D(G)\geq (2n-4)/3$,  and for any
  closed quipu $G$ on $n$ vertices ($n\geq 13$) with spectral radius less than
  $\frac{3}{2}{\sqrt{2}}$, its diameter $D(G)$ satisfies $\frac{n}{3}< D(G)\leq
 \frac{2n-2}{3}$.

In this paper, we would like to study the $r$-uniform hypergraphs $H$ with
small spectral radius.
In our previous paper \cite{LM}, we generalized Smith's theorem to hypergraphs
and  classified all connected $r$-uniform hypergraphs with the spectral
radius at most $\rho_r=(r-1)!\sqrt[r]{4}$. The main method is using
{\em $\alpha$-normal labeling}. Roughly speaking, we can label all ``corners
of edges'' by some numbers in $(0,1)$ such that for each vertex $v$
the sum of these numbers at $v$ is always equal to $1$ while for each
edge $f$ the product of these numbers at $f$ is always equal to
$\alpha$. The detail of the definition of $\alpha$-normal labeling
can be found in Section 2. If $H$ has a ``consistent'' $\alpha$-normal
labeling, then $\rho(H)=(r-1)!\alpha^{-1/r}$. As an important
corollary, any $(r-1)$-uniform hypergraph $H'$ with
$\rho(H')=(r-2)!\alpha^{-1/(r-1)}$ can be extended to an $r$-uniform
hypergraph $H$ with spectral radius $\rho(H)=(r-1)!\alpha^{-1/r}$ by
simply extending each edge by adding one new vertex.
If $H$ is not extended from some $H'$, then $H$ is called {\em irreducible}.
An $r$-uniform hypergraph is irreducible if and only if it contains an
edge so that every vertex in this edge has degree greater than 1.
We use the following convention: if the notation $H^{(r')}$ is
a well-defined $r'$-uniform hypergraph, then for each $r>r'$, $H^{(r)}$ means the unique
$r$-uniform hypergraph extended from $H^{(r')}$ by a sequence of
extension described above.

From \cite{LM}, we show all $r$-uniform hypergraphs $H$ with
$\rho(H)= (r-1)!\sqrt[r]{4}$ listed as follows:
\begin{description}
\item[Extended from $2$-graphs:] $C_{n}^{(r)}$, $\tilde D_{n}^{(r)}$,
$\tilde E_{6}^{(r)}$,  $\tilde E_{7}^{(r)}$, and $\tilde E_{8}^{(r)}$.

\item[Extended from $3$-graphs:]  $\tilde{B}_{n}^{(r)}$, $\widetilde{BD}_{n}^{(r)}$, $C_{2}^{(r)}$, $S^{(r)}_{4}$,  $F_{2,3,4}^{(r)}$, $F_{2,2,7}^{(r)}$, $F_{1,5,6}^{(r)}$, $F_{1,4,8}^{(r)}$,  $F_{1,3,14}^{(r)}$,  $G_{1,1:0:1,4}^{(r)}$, and $G_{1,1:6:1,3}^{(r)}$.
\item[Extended from $4$-graphs:] $H^{(r)}_{1,1,2,2}$.
\end{description}
Similarly here are all $r$-uniform hypergraphs $H$ with
$\rho(H)< (r-1)!\sqrt[r]{4}$:
\begin{description}
\item[Extended from $2$-graphs:] $A_{n}^{(r)}$, $D_{n}^{(r)}$,
$E_{6}^{(r)}$,  $E_{7}^{(r)}$, and $E_{8}^{(r)}$.

\item[Extended from $3$-graphs:] ${D'}_{n}^{(r)}$, $B_{n}^{(r)}$, ${B'}_{n}^{(r)}$,
$\bar{B}_{n}^{(r)}$, $BD_{n}^{(r)}$,  $F_{2,3,3}^{(r)}$, $F_{2,2,j}^{(r)}$ (for $2\leq j\leq6$), $F_{1,3,j}^{(r)}$ (for $3\leq j\leq 13$), $F_{1,4,j}^{(r)}$ (for $4\leq j\leq 7$),  $F_{1,5,5}^{(r)}$, and  $G_{1,1:j:1,3}^{(r)}$ (for $0\leq j\leq 5$).

\item[Extended from $4$-graphs:] $H^{(r)}_{1,1,1,1}$, $H^{(r)}_{1,1,1,2}$, $H^{(r)}_{1,1,1,3}$,
  $H^{(r)}_{1,1,1,4}$.
\end{description}
The details of these hypergraphs  can be found in the paper \cite{LM}.

%  For $r\geq 3$, we say an $r$-uniform hypergraph $H$ is {\em
%    reducible} if every edge of $H$ contains a leaf vertex.
% By deleting one leaft vertex from every edge, we obtained an
%  $(r-1)$-uniform hypergraph $H'$. In this case, we say $H$ {\em extends}
%  $H'$ or $H'$ is reduced from $H$. An important fact observed in
% \cite{speclimgh} is that
% if $H$ extends $H'$ then $\rho(H)\leq (r-1)!\alpha^{-1/r}$ for some
% $\alpha\in (0,1)$ if and only if $\rho(H')\leq (r-2)!\alpha^{-1/(r-1)}$.
% Thus, it is suffice to consider the irreducible hypergraphs.
%Our main result is the following quipus structures.

It is nature to ask what structures the
hypergraphs with spectral radius slightly greater than $\rho_r$ can
have. Since $(2,\sqrt{2+\sqrt{5}})$ is the next interesting interval
for the spectral radius of graphs, naturally we consider all connected $r$-uniform hypergraphs $H$ with $\rho(H)\in
((r-1)!\sqrt[r]{4}, (r-1)!\sqrt[r]{2+\sqrt{5}})$. When $r=2$, these
graphs are $E_{1,b,c}$ ,
$E_{2,2,c}$, and $G_{1,a:b:1,c}$ with $b>a+c$ as shown by
Cvetkovi\'c et al.~\cite{CDG} and  Brouwer-Neumaier
\cite{BN}. The structures of these hypergraphs are slightly more complicated
for $r\geq 3$. For $k\geq 3$, a vertex is called a {\em $k$-branching vertex} if it is incident
to $k$ edges  while an edge is called a {\em $k$-branching edge}
if it contains no branching vertex but it is adjacent to exactly $k$
edges. (When $k=3$, we simply say branching vertex/edge instead of
$3$-branching vertex/edge.)
We have the following results.

\begin{theorem}\label{t1}
Consider an irreducible connected $3$-uniform hypergraph $H$. If
the spectral radius of $H$ satisfies
$\rho(H)\leq 2\sqrt[3]{2+\sqrt{5}}$, then no vertex (of $H$) can have
degree more than three, no edge can incident to more than $3$ other
edges, each branching vertex is not incident to any branching edges.
Moreover, $H$ belongs to one of the following two categories:
\begin{description}
\item[Open $3$-quipu:]
$H$ is a hypertree with all branching vertices and all branching edges
lying on a path. Moreover, there are at most 2 branching vertices.
A branching vertex cannot lie between two branching edges, or between
a branching edge and another branching vertex.
\item[Closed $3$-quipu:] $H$ contains a cycle $C$ and no branching vertices.
All branching edges lie on $C$, and any branching edge can be only attached by a path.
%\item There exists a path $P$ containing all branching
%  vertices and branching edges such that for each end of $P$ either
%the ending edge or the ending vertex is branching. Moreover,
%the branching vertex/vertices, if exists, must have degree 3 and can only locate at the ends of $P$.
%\item There exists a circle $C$ containing  all branching
%  edges. In this case, $H$ contains no branching vertex.
\end{description}
\end{theorem}
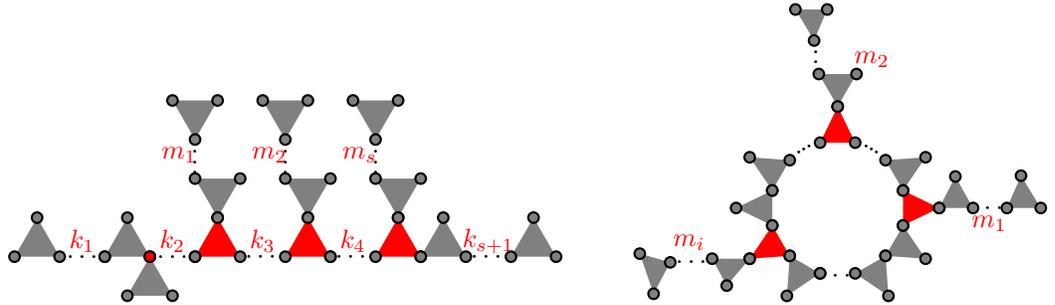
\begin{figure}[hbt]
 \begin{center}
\begin{tikzpicture}[thick, scale=0.6, bnode/.style={circle, draw,
    fill=black!50, inner sep=0pt, minimum width=4pt},
  enode/.style={color=red}, rnode/.style={circle, draw,
    fill=red, inner sep=0pt, minimum width=4pt}]
\foreach \x in {4,6,8}
    {
\path[fill=gray]  (\x+0.5,0.866) node [bnode] {}-- (\x,1.732) node [bnode] {} --
    (\x+1,1.732)  node [bnode] {} --cycle;
  \draw (\x,2.28) node [color=black] {$\vdots$};

\path[fill=gray]  (\x,2.598) node [bnode] {}-- (\x-0.5,3.464) node [bnode] {} --
    (\x+0.5,3.464)  node [bnode] {} --cycle;

}
\foreach \x in {4,6,8}
    {
    \path[fill=red]  (\x,0) node [bnode] {} -- (\x+0.5,0.866) node [bnode] {} --(\x+1,0)node [bnode] {}--cycle;
}
\foreach \x in {0,2,9,11}
    {
    \path[fill=gray]  (\x,0) node [bnode] {} -- (\x+0.5,0.866) node [bnode] {} --(\x+1,0)node [bnode] {}--cycle;
}

 \path[fill=gray]  (3,0)  -- (2.5,-0.866) node [bnode] {}
 --(3.5,-0.866) node [bnode] {}--cycle;
\draw (3,0) node [rnode] {};

\draw (5.5,0) node [] {$\cdots$};
\draw (7.5,0) node [] {$\cdots$};
\draw (2,0) node  [bnode] {};
\draw (4,0) node  [bnode] {};
\draw (7,0) node  [bnode] {};

\draw (1.5,0) node []{$\ldots$};
\draw (3.5,0) node []{$\ldots$};
\draw (3.65,2.25) node [enode] {$m_1$};
\draw (5.65,2.25) node [enode] {$m_2$};
\draw (7.65,2.25) node [enode] {$m_s$};
%\draw (3.55,1.25) node [enode] {$\uparrow$};
%\draw (5.75,1.25) node [enode] {$\uparrow$};
%\draw (7.7,1.25) node [enode] {$\uparrow$};
\draw (10.5,0) node [] {$\cdots$};
\draw (1.5,0.35) node [enode] {$k_1$};
\draw (3.5,0.35) node [enode] {$k_2$};
\draw (5.5,0.35) node [enode] {$k_3$};
\draw (7.5,0.35) node [enode] {$k_4$};
\draw (10.5,0.35) node [enode] {$k_{s+1}$};
\end{tikzpicture}
\hfil
\begin{tikzpicture}[thick, scale=0.45, bnode/.style={circle, draw,
    fill=black!50, inner sep=0pt, minimum width=4pt}, enode/.style={color=red}]
\foreach \x in {-60,-30,30,150,180,240}
    {
    \path[fill=gray]  (\x-15:2) node [bnode] {} -- (\x:3) node [bnode] {} --(\x+15:2) node [bnode] {}--cycle;
}
\foreach \x in {0, 90,210}
    {
    \path[fill=red]  (\x-15:2) node [bnode] {} -- (\x:3) node [bnode] {} --(\x+15:2) node [bnode] {}--cycle;
}
\draw (270:2) node  [color=black] {$\cdots$};
\draw (4.5,0) node  [color=black] {$\cdots$};
\draw (255:2) node [bnode] {};
\draw (60:2) node [color=black] {$\cdot$};
\draw (55:2) node [color=black] {$\cdot$};
\draw (65:2) node [color=black] {$\cdot$};
\draw (120:2) node [color=black] {$\cdot$};
\draw (125:2) node [color=black] {$\cdot$};
\draw (115:2) node [color=black] {$\cdot$};
\draw (98:4.6) node [color=black] {$\vdots$};
\path[fill=gray]  (0:5) node [bnode] {} -- (10:5.5) node [bnode] {} --(0:6)node [bnode] {}--cycle;
\path[fill=gray]  (0:3) node [bnode] {} -- (15:3.6) node [bnode] {} --(0:4)node [bnode] {}--cycle;
\path[fill=gray]  (98:5) node [bnode] {} -- (102:6) node [bnode] {} --(93:5.877)node [bnode] {}--cycle;
\path[fill=gray]  (90:3) node [bnode] {} -- (98:4) node [bnode] {} --(82:4)node [bnode] {}--cycle;
\path[fill=gray]  (210:3) node [bnode] {} -- (202:4) node [bnode] {} --(215:4)node [bnode] {}--cycle;
\draw (200:4.6) node [color=black] {$\ldots$};
\path[fill=gray]  (198:5.2) node [bnode] {} -- (193:6.0) node [bnode] {} --(205:6.0)node [bnode] {}--cycle;
\draw (-6:4.5) node [enode] {$m_{1}$};
\draw (77:4.5) node [enode] {$m_{2}$};
\draw (193:4.5) node [enode] {$m_{i}$};
\end{tikzpicture}
  \caption{(Examples) Left: an open $3$-quipu where the branching vertex/edges
    are filled in red.
 Right: a closed quipu where the branching edges are
 filled in red. }
  \label{fig:1 }
\end{center}
\end{figure}
\begin{theorem}\label{t2}
Suppose that $H$ is an irreducible $4$-uniform hypergraphs with
$\rho(H)\leq 6\sqrt[4]{2+\sqrt{5}}$. Then $H$ is a hypertree with
no vertex (of $H$) having
degree more than three and no edge incident to more than $4$ other
edges. The hypergraph $H$ belongs to one of the following two categories:
\begin{description}
\item[Open 4-quipu:]
$H$ is a hypertree with all branching vertices and all branching edges
lying on a path. Moreover, there are at most two $3$-branching vertices (or two  $4$-branching edges).
A $4$-branching edge (or a branching vertex) cannot lie between two $3$-branching edges, or between
a $3$-branching edge and another $4$-branching edge (or a branching
vertex). In addition, each $4$-branching edge is attached by three
path of length $1$, $1$, and $k$ ($k=1,2,3$) respectively.
\item [$4$-dagger:] $H$ is obtained by attaching $4$-paths of length
  $i,j,k,l$ to a $4$-branching edge. Denote this hypergraph by
  $H^{(4)}_{i,j,k,l}$ with $i\leq j\leq k\leq l$. Then $H$ must be one
  of the following hypergraphs $H^{(4)}_{1,2,2,2}$,
  $H^{(4)}_{1,2,2,3}$, $H^{(4)}_{1,1,4,4}$, $H^{(4)}_{1,1,4,5}$, and
$H^{(4)}_{1,1,k,l}$ ($1\leq k\leq 3$, and $k\leq l$).
 \end{description}
 \end{theorem}
 \begin{figure}[hbt]
  \begin{center}
\begin{tikzpicture}[thick, scale=0.5, bnode/.style={circle, draw,
    fill=black!50, inner sep=0pt, minimum width=4pt}, enode/.style={red}]

\path[fill=gray]  (1.5,0.5)node [bnode] {}  -- (1,1) node [bnode] {} --(1.5,1.5)node
[bnode] {}-- (2,1) node [bnode]{} --cycle;
\path[fill=gray]  (1.5,-0.5)node [bnode] {}  -- (1,-1) node [bnode] {} --(1.5,-1.5)node [bnode] {}-- (2,-1) node [bnode]{} --cycle;

\foreach \x in {-2,0,1,2,4,6,8}
    {
    \path[fill=gray]  (\x,0) node [bnode] {} -- (\x+0.5,0.5) node [bnode] {} --(\x+1,0)node [bnode] {}-- (\x+0.5,-0.5) node [bnode]{} --cycle;

}
\draw (2,0) node  [bnode] {};
\draw (-0.5,0) node [color=black] {$\cdots$};
\draw (3.5,0) node [color=black] {$\cdots$};
\draw (5.5,0) node [color=black] {$\cdots$};
\draw (7.5,0) node [color=black] {$\cdots$};
\draw (9.5,0) node [color=black] {$\cdots$};
\foreach \x in {4.5,6.5,8.5}
{\path[fill=gray]  (\x,2.5)node [bnode] {}  -- (\x-0.5,3) node [bnode] {} --(\x,3.5)node
[bnode] {}-- (\x+0.5,3) node [bnode]{} --cycle;

\path[fill=gray]  (\x,0.5)node [bnode] {}  -- (\x-0.5,1) node [bnode] {} --(\x,1.5)node
[bnode] {}-- (\x+0.5,1) node [bnode]{} --cycle;
\draw (\x,2.2) node [color=black] {$\vdots$};
\draw (\x-0.7,1.35) node [enode] {$\uparrow$};
}
\draw (3.25,1.7) node [enode] {$m_{1}$};
\draw (5.25,1.7) node [enode] {$m_{2}$};
\draw (7.25,1.7) node [enode] {$m_{i}$};
\foreach \x in {-2,0,1,2,4}
\draw (0.35,0.75) node [enode] {$\leftarrow$};
\draw (0.5,1.2) node [enode] {$k$};
\draw (3.5,0.35) node [enode] {$k_{1}$};
\draw (5.5,0.35) node [enode] {$k_{2}$};
\draw (7.5,0.35) node [enode] {$k_{i}$};
\draw (10,0) node [bnode] {};
\path[fill=gray]  (10,0)  -- (10.66,-0.30) node [bnode] {} --(10.56,-0.86)node
[bnode] {}-- (9.90,-0.68) node [bnode]{} --cycle;
\path[fill=gray]  (10,0)  -- (10.66,0.30) node [bnode] {} --(10.56,0.86)node
[bnode] {}-- (9.90,0.68) node [bnode]{} --cycle;
\end{tikzpicture}
\hfil
\begin{tikzpicture}[thick, scale=0.5, bnode/.style={circle, draw,
    fill=black!50, inner sep=0pt, minimum width=4pt}, enode/.style={red}]

\path[fill=gray]  (1.5,2.5)node [bnode] {}  -- (1,3) node [bnode] {} --(1.5,3.5)node
[bnode] {}-- (2,3) node [bnode]{} --cycle;
\path[fill=gray]  (1.5,-2.5)node [bnode] {}  -- (1,-3) node [bnode] {} --(1.5,-3.5)node [bnode] {}-- (2,-3) node [bnode]{} --cycle;
\path[fill=gray]  (1.5,0.5)node [bnode] {}  -- (1,1) node [bnode] {} --(1.5,1.5)node
[bnode] {}-- (2,1) node [bnode]{} --cycle;
\path[fill=gray]  (1.5,-0.5)node [bnode] {}  -- (1,-1) node [bnode] {} --(1.5,-1.5)node [bnode] {}-- (2,-1) node [bnode]{} --cycle;
\foreach \x in {-2,0,1,2,4}
    {
    \path[fill=gray]  (\x,0) node [bnode] {} -- (\x+0.5,0.5) node [bnode] {} --(\x+1,0)node [bnode] {}-- (\x+0.5,-0.5) node [bnode]{} --cycle;

}
\draw (2,0) node  [bnode] {};
\draw (-0.5,0) node [color=black] {$\cdots$};
\draw (1.5,2.2) node [color=black] {$\vdots$};
\draw (1.5,-1.85) node [color=black] {$\vdots$};
\draw (3.5,0) node [color=black] {$\cdots$};
\draw (0.5,1.2) node [enode] {$i$};
\draw (0.8,1.2) node [enode] {$\uparrow$};
\draw (0.35,-0.75) node [enode] {$\leftarrow$};
\draw (0.45,-1.0) node [enode] {$j$};
\draw (2.5,-1.2) node [enode] {$l$};
\draw (2.2,-1.2) node [enode] {$\downarrow$};
\draw (2.7,0.7) node [enode] {$\rightarrow$};
\draw (2.7,1.0) node [enode] {$k$};
\end{tikzpicture}
 \caption{(Examples) Left: an open $4$-quipu.
 Right: a $4$-dagger $H^{(4)}_{i,j,k,l}$. }
  \label{fig:2}
\end{center}
\end{figure}
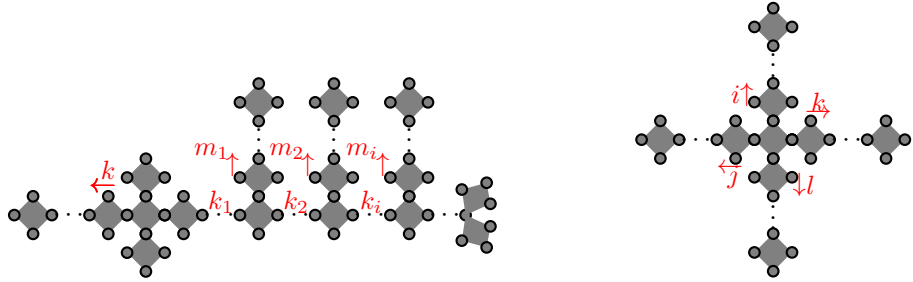
\begin{theorem}\label{t3}
For $r=5$, there is only one irreducible $5$-uniform hypergraph $H$
with $\rho(H)\leq (r-1)!\sqrt[r]{2+\sqrt{5}}$; namely the five edge-star as
shown below.
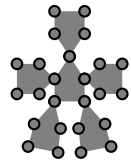
\begin{center}
\begin{tikzpicture}[thick, scale=0.4, bnode/.style={circle, draw,
    fill=black!50, inner sep=0pt, minimum width=4pt}, enode/.style={red}]
  \path[fill=gray]  (0,0) node [bnode] {} -- (0.5,0.75) node [bnode] {} --(1,0)node [bnode] {}-- (1,-0.75) node [bnode]{}-- (0,-0.75) node [bnode]{}--cycle;
  \path[fill=gray]  (0.5,0.75) node [bnode] {} -- (0,1.5) node [bnode] {} --(0,2.25)node [bnode] {}-- (1,2.25) node [bnode]{}-- (1,1.5) node [bnode]{}--cycle;
 \path[fill=gray]  (0,0) node [bnode] {} -- (-0.5,0.5) node [bnode] {} --(-1.25,0.5)node [bnode] {}-- (-1.25,-0.5) node [bnode]{}-- (-0.5,-0.5) node [bnode]{}--cycle;
 \path[fill=gray]  (0,-0.75) node [bnode] {} -- (-0.68,-1.5) node [bnode] {} --(-0.88,-2.25)node [bnode] {}-- (0.12,-2.45) node [bnode]{}-- (0.2,-1.65) node [bnode]{}--cycle;
  \path[fill=gray]  (1,0) node [bnode] {} -- (1.5,0.5) node [bnode] {} --(2.25,0.5)node [bnode] {}-- (2.25,-0.5) node [bnode]{}-- (1.5,-0.5) node [bnode]{}--cycle;
 \path[fill=gray]  (1,-0.75) node [bnode] {} -- (0.8,-1.65) node [bnode] {} --(0.88,-2.45)node [bnode] {}-- (1.88,-2.25) node [bnode]{}-- (1.68,-1.5) node [bnode]{}--cycle;
\end{tikzpicture}
\end{center}
For $r\geq 6$, all $r$-uniform hypergraphs $H$
with $\rho(H)\leq (r-1)!\sqrt[r]{2+\sqrt{5}}$ are reducible.
\end{theorem}

\section{Notation and Lemmas}
Let us review some basic notation about hypergraphs.
An $r$-uniform hypergraph $H$ is a pair $(V,E)$ where $V$ is the set of
vertices and $E\subset {V\choose r}$ is the set of edges.
The degree of vertex $v$, denoted by $d_v$, is the number of edges incident to $v$. If
$d_v=1$, we say $v$ is a leaf vertex.
A {\em walk} on a hypergraph $H$ is a sequence of vertices and edges:
$v_0e_1v_1e_2\ldots v_l$ satisfying that both $v_{i-1}$ and $v_i$
are incident to $e_i$ for $1\leq i \leq l$. The vertices $v_0$ and
$v_l$ are called the ends of the walk. The length of a walk is the
number of edges on the walk.
A walk is called a {\em path} if
all vertices and edges on the walk are distinct.
The walk is {\em closed}
if $v_l=v_0$. A closed walk is called a {\em cycle} if all vertices and
edges in the walk are distinct. A hypergraph $H$ is called {\em
  connected} if for any pair of vertex $(u,v)$, there is a path
connecting $u$ and $v$. A hypergraph $H$ is called a {\em hypertree} if it is
connected, and acyclic.
 A hypergraph
$H$ is called {\em simple} if every pair of edges intersects at most
one vertex.  In fact, any non-simple hypergraph
contains at least a $2$-cycle: $v_1F_1v_2F_2v_1$, i.e., $v_1,v_2\in F_1\cap
F_2$. A hypertree is always simple.

The spectral radius $\rho(H)$ of an $r$-uniform hypergraph $H$ is defined as
\begin{equation}
  \label{eq:1}
\rho(H)=r!\max_{\stackrel{{\bf x}\in {\mathbb R}^n_{\geq 0}}{{\bf x}
        \not=0}}
\frac{\sum_{\{i_1, i_2,\cdots, i_r\}\in E(H)} x_{i_1}x_{i_2}\cdots
x_{i_r }}{\sum_{i=1}^n x_i^r}.
\end{equation}
Here ${\mathbb R}^n_{\geq 0}$ denote the set of points with
nonnegative coordinates in  ${\mathbb R}^n$.
 This is a special case
of $p$-spectral norm for $p=r$. The general $p$-spectral norm has been considered by various authors
(see \cite{SSL, CD, KLM, nikiforov}).  The following lemma has been
proved in several papers.
\begin{lemma}
\cite{CD, KLM, nikiforov}
\label{subgraph}
If $G$ is a connected $r$-uniform hypergraph, and $H$ is a proper
subgraph of $G$, then
$$\rho(H)<\rho(G).$$
\end{lemma}

In our previous paper \cite{LM}, we discovered an efficient way
to compute the spectral radius $\rho(H)$, in particular when $H$ is a
hypertree. The idea is using the method of the
$\alpha$-normal labelling (or weighed matrix).

\begin{definition}[See \cite{LM}]
A weighted incidence matrix $B$ of a hypergraph $H$ is a $|V|\times |E|$
matrix such that for any vertex $v$ and any edge $e$,  the entry
$B(v,e)>0$ if $v\in e$ and $B(v,e)=0$ if $v\not\in e$.
\end{definition}

\begin{definition}[See \cite{LM}]
\label{anormal}
A hypergraph $H$ is called $\alpha$-normal if there exists a weighted
incidence matrix $B$ satisfying
\begin{enumerate}
\item $\sum_{e\colon v\in e}B(v,e)=1$, for any  $v\in V(H)$.
\item $\prod_{v\in e}B(v,e)=\alpha$,  for any $e\in E(H)$.
\end{enumerate}
Moreover, the incidence matrix $B$ is called {\em consistent}
if for any cycle $v_0e_1v_1e_2\ldots v_l$ ($v_l=v_0$)
$$\prod_{i=1}^l\frac{B(v_{i},e_i)}{B(v_{i-1}, e_i)}=1.$$
In this case, we call $H$ {\em consistently $\alpha$-normal}.
 \end{definition}
The following important lemma was proved in \cite{LM}.
\begin{lemma}[See Lemma 3 of \cite{LM}] \label{l:main}
 Let $H$ be a connected $r$-uniform hypergraph.
 Then the spectral radius of $H$ is
 $\rho(H)$ if and only if  $H$ is consistently $\alpha$-normal with $\alpha=((r-1)!/\rho(H))^r$.
\end{lemma}

Often we need compare the spectral radius with a particular value.
 \begin{definition}[See \cite{LM}]
   A hypergraph $H$ is called $\alpha$-subnormal if there exists a weighted
incidence matrix $B$ satisfying
\begin{enumerate}
\item $\sum_{e\colon v\in e}B(v,e)\leq 1$, for any  $v\in V(H)$.
\item $\prod_{v\in e}B(v,e)\geq \alpha$,  for any $e\in E(H)$.
\end{enumerate}
Moreover, $H$ is called {\em strictly  $\alpha$-subnormal} if
it is $\alpha$-subnormal but not $\alpha$-normal.
 \end{definition}
We have the following lemma.
\begin{lemma}[See Lemma 4 of \cite{LM}] \label{subnormal}
Let $H$ be an $r$-uniform hypergraph. If $H$ is $\alpha$-subnormal, then
the spectral radius of $H$ satisfies
\begin{equation*}
 \rho(H)\leq (r-1)!\alpha^{-\frac{1}{r}}.
\end{equation*}
Moreover, if $H$ is strictly $\alpha$-subnormal
then $\rho(H)< (r-1)!\alpha^{-\frac{1}{r}}$.
\end{lemma}

\begin{definition}[See \cite{LM}]
   A hypergraph $H$ is called $\alpha$-supernormal if there exists a weighted
incidence matrix $B$ satisfying
\begin{enumerate}
\item $\sum_{e\colon v\in e}B(v,e)\geq 1$, for any  $v\in V(H)$.
\item $\prod_{v\in e}B(v,e)\leq \alpha$,  for any $e\in E(H)$.
\end{enumerate}
Moreover, $H$ is called {\em strictly  $\alpha$-supernormal} if
it is $\alpha$-supernormal but not $\alpha$-normal.
 \end{definition}
\begin{lemma}[See Lemma 5 of \cite{LM}] \label{supernormal}
Let $H$ be an $r$-uniform hypergraph. If $H$ is strictly and consistently $\alpha$-supernormal, then
the spectral radius of $H$ satisfies
\begin{equation*}
 \rho(H)> (r-1)!\alpha^{-\frac{1}{r}}.
\end{equation*}
\end{lemma}

Note that if $H$ is consistently $\alpha$-normal and $H$ is extended
from $H'$, then so is $H'$. This implies the following corollary.

\begin{corollary}
\label{extend}
 For any $r\geq 3$ and $\alpha\in(0,1)$, if $H$ extends $H'$, then $\rho(H)=(r-1)!\alpha^{-1/r}$
 (or  $\rho(H)<(r-1)!\alpha^{-1/r}$) if
  and only if $\rho(H')=(r-2)!\alpha^{-1/(r-1)}$ (or
  $\rho(H')<(r-2)!\alpha^{-1/(r-1)}$ respectively).
\end{corollary}

\begin{definition}
  Given two $r$-uniform hypergraphs $H_1$ and $H_2$, a  {\em
    homomorphism} from $H_1$ to $H_2$ is a map $f\colon V(H_1)\to
  V(H_2)$ which preserves the edges. If $f$ derives an injective map,
  also denoted by $f$,
 from $E(H_1)$ to $E(H_2)$, then $f$ is called a {\em
   sub-homomorphism}. In this case, we also say $H_1$ is a {\em
   sub-homomorphic type} of $H_2$.
\end{definition}

Every subhypergraph is a subhomorphic type. The reverse statement is not
true. Consider the following example. Suppose that $v_1$ and $v_2$ are
two vertices of $H$ which are not contained in any common edge. We can
form  a new  hypergraph $H'$ from $H$ by identifying $v_1$ and $v_2$
into a fat vertex, called $x$. Now the map $f\colon V(H)\to V(H')$
by sending both $v_1$ and $v_2$ into $x$ and mapping other vertices
itself. Then $f$ is a sub-homomorphism. The following lemma
generalizes Lemma \ref{subgraph}.

\begin{lemma}\label{subhomo}
Suppose $H_1$ and $H_2$ are two connected  $r$-uniform hypergraphs.
If $H_1$ is a sub-homomorphic type of $H_2$, then
we have
$$\rho(H_1)\leq  \rho(H_2)$$
and the equaility holds if and only if $H_1$ is isomorphic to $H_2$.
\end{lemma}
\begin{proof}
Let $f\colon V(H_1)\to V(H_2)$ be the sub-homomorphism.
Setting $\alpha=\left(\frac{(r-1)!}{\rho(H_2)}\right)^r$, by Lemma
\ref{anormal}, $H_2$ is consistently $\alpha$-normal and let $B_2$ be
the incident matrix. We can define an incident matrix $B_1$ of $H_1$
as follows:
$$B_1(v,e)=B_2(f(v), f(e)) \quad \mbox{ for any } v\in V(H_1) \mbox{
  and } e\in E(H_1).$$
For any fixed $e\in E(H_1)$, we have $$\prod_{v\in e}B_1(v,e)=\prod_{v'\in
  f(e)}B_2(v',f(e))=\alpha.$$
For any fixed $v\in E(H_1)$, the set $\{e\in E(H_1)\colon v\in e\}$ is
a subset of $\{e\in E(H_1)\colon f(v)\in f(e)\}$.
Since $f(e)$ is uniquely determined by $e$, the latter set
is one-to-one corresponding to the set  $\{e'\in E(H_2)\colon f(v)\in e'\}$.
 This observation implies
 $$\sum_{e\colon v\in e}B_1(v,e)\leq \sum_{e'\colon f(v)\in e'} B_2(f(v),e')=1.$$
Therefore, $H_1$ is $\alpha$-subnormal. It implies $\rho(H_1)\leq
\rho(H_2)$. When the inequality holds, $f(H_1)=H_2$ (otherwise
$\rho(H_1)\leq \rho(f(H_1))<\rho(H_2)$),
and
for any $v\in V(H_1)$ and $e\in
E(H_1)$, $v\in e$ if and only if $f(v)\in f(e)$. This implies that $f$ must be an
injective map, (otherwise, we have $f(v_1)=f(v_2)$, then we can find an
edge $e_1$ containing $v_1$. Since $f$ is a homomorphism, $v_2$ is not
in $e_1$, but $f(v_2)=f(v_1)\in f(e_1)$. Contradiction.) Hence, $f$ is
an isomorphism.
\end{proof}

Often, we need to calculate the limit of the spectral radius  of a
sequence of hypergraphs. The following lemma is helpful.

\begin{lemma}\label{infty}
For any fixed $\beta\in (0,\frac{1}{4})$, let $f_{\beta}(x)=\frac{\beta}{1-x}$  and $f^{n}_{\beta}(x)=f(f_{\beta}^{n-1}(x))$ for $n\geq 2$.
\begin{enumerate}
\item If $0<x\leq \frac{1-\sqrt{1-4\beta}}{2}$, then $f^{n}_{\beta}(x)$ is increasing with respect to $n$, and $\lim_{n\rightarrow \infty}f^{n}_{\beta}(x)=\frac{1-\sqrt{1-4\beta}}{2}.$ Moreover, when $x=\frac{1-\sqrt{1-4\beta}}{2}$,  $f^{n}_{\beta}(x)=\frac{1-\sqrt{1-4\beta}}{2}, \forall n\geq 1.$
\item If $\frac{1-\sqrt{1-4\beta}}{2}<x<\frac{1+\sqrt{1-4\beta}}{2}$, then $f^{n}_{\beta}(x)$ is decreasing with respect to $n$, and $\lim_{n\rightarrow \infty}f^{n}_{\beta}(x)=\frac{1-\sqrt{1-4\beta}}{2}.$
\end{enumerate}
\end{lemma}
\begin{proof}
We first prove item 1. Since $0<x\leq \frac{1-\sqrt{1-4\beta}}{2}$, the function $f_{\beta}(x)=\frac{\beta}{1-x}$ attains its
maximum when $x= \frac{1-\sqrt{1-4\beta}}{2}$. So, $0<f_{\beta}(x)\leq \frac{1-\sqrt{1-4\beta}}{2}$. Similarly,
$f^{2}_{\beta}(x)=\frac{\beta}{1-f_{\beta}(x)}$ attains its maximum when $f_{\beta}(x)=\frac{1-\sqrt{1-4\beta}}{2}$, so we get
$0<f^{2}_{\beta}(x)\leq \frac{1-\sqrt{1-4\beta}}{2}$. With the same way, we get $0<f^{n}_{\beta}(x)\leq \frac{1-\sqrt{1-4\beta}}{2}$, for all $n\geq 3$.
On the other hand, if $0<f^{n}_{\beta}(x)<\frac{1-\sqrt{1-4\beta}}{2}$, we can easily check that $f^{n}_{\beta}(x)-f^{n-1}_{\beta}(x)=\frac{\beta}{1-f^{n-1}_{\beta}(x)}-f^{n-1}_{\beta}(x)
=\frac{\beta-f^{n-1}_{\beta}(x)+(f^{n-1}_{\beta}(x))^{2}}{1-f^{n-1}_{\beta}(x)}>0$ for all $n\geq 2$.
So, $f^{n-1}_{\beta}(x)<f^{n}_{\beta}(x)$ for all $n\geq 2$. Thus, we let $\lim_{n\rightarrow \infty}f^{n}_{\beta}(x)=f_{0}(x)$, and by $f^{n}_{\beta}(x)=\frac{\beta}{1-f^{n-1}_{\beta}(x)}$, we get
$f_{0}(x)=\frac{1-\sqrt{1-4\beta}}{2}$. The proof of item 2 is very similar to the proof of item 1, so we omit the proof here.
\end{proof}
\begin{lemma}\label{circle}
 Let $f_{\beta}(x)=\frac{\beta}{1-x}$  and
 $f_{\beta}^{n}(x)=f(f_{\beta}^{n-1}(x))$ for $n\geq 2$, then for any
 positive integer $n$, and any real $\beta \in (0,\frac{1}{4})$, there exists a unique
 $x\in(\frac{1-\sqrt{1-4\beta}}{2}, \frac{1+\sqrt{1-4\beta}}{2})$ such that $f_{\beta}^{n}(x)=1-x$.
\end{lemma}
\begin{proof}
Consider the set ${\cal F}$ of functions $f$ satisfying
\begin{enumerate}
\item $f$ is an increasing continuous function in
$(\frac{1-\sqrt{1-4\beta}}{2}, \frac{1+\sqrt{1-4\beta}}{2})$.
\item Both
$\frac{1-\sqrt{1-4\beta}}{2}$ and $\frac{1+\sqrt{1-4\beta}}{2}$ are
 fixed points of $f$.
\end{enumerate}
We claim that for any $f\in {\cal F}$ there exists a unique
 $x\in(\frac{1-\sqrt{1-4\beta}}{2}, \frac{1+\sqrt{1-4\beta}}{2})$ such that $f(x)=1-x$.
This is because $g(x):=f(x)+x$ is a strictly increasing and continuous function in
$(\frac{1-\sqrt{1-4\beta}}{2}, \frac{1+\sqrt{1-4\beta}}{2})$ and
$$g(\frac{1-\sqrt{1-4\beta}}{2})=1-\sqrt{1-4\beta}<1, \quad
\mbox{ and } g(\frac{1+\sqrt{1-4\beta}}{2})=1+\sqrt{1-4\beta}>1.$$

It suffices to show $f^m_{\beta}(x)\in {\cal F}$ for any positive integer
$m$. This can be proved by induction on $m$. For $m=1$,
$f^1_\beta(\beta)=f_\beta(\beta)\in {\cal F}$ can be easily verified. Now we assume
$f^m_\beta\in {\cal F}$.
Note both $f_\beta$ and $f^m_\beta$ map $(\frac{1-\sqrt{1-4\beta}}{2},
\frac{1+\sqrt{1-4\beta}}{2})$
to $(\frac{1-\sqrt{1-4\beta}}{2}, \frac{1+\sqrt{1-4\beta}}{2})$
increasingly
and continuously to itself. So is their composition, $f_\beta\circ f^m_\beta=f^{m+1}_\beta$.
We finished the proof.
\end{proof}

\begin{lemma}
Let the following graph denote $F^{(3)}_{m,n,k}$,
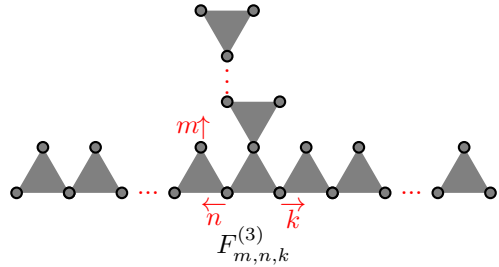
\begin{center}
\begin{tikzpicture}[thick, scale=0.7, bnode/.style={circle, draw,
    fill=black!50, inner sep=0pt, minimum width=4pt}, enode/.style={color=red}]
\path[fill=gray]  (4.5,0.866) -- (4,1.732) node [bnode] {} --
    (5,1.732)  node [bnode] {} --cycle;
\path[fill=gray]  (4,2.598) node [bnode] {}-- (3.5,3.464) node [bnode] {} --
    (4.5,3.464)  node [bnode] {} --cycle;

\foreach \x in {0,1,3,4,5,6,8}
    {
    \path[fill=gray]  (\x,0) node [bnode] {} -- (\x+0.5,0.866) node [bnode] {} --(\x+1,0)node [bnode] {}--cycle;
}
\draw (2.5,0) node [enode]{$...$};
\draw (7.5,0) node [enode]{$...$};

\draw (4,2.265) node [enode]{$\vdots$};

\draw (3.25,1.25) node [enode] {$m$};
\draw (3.55,1.25) node [enode] {$\uparrow$};
\draw (3.75,-0.25) node [enode] {$\leftarrow$};
\draw (3.75,-0.5) node [enode] {$n$};
\draw (5.25,-0.25) node [enode] {$\rightarrow$};
\draw (5.25,-0.5) node [enode] {$k$};
\draw (4.5, -1) node [color=black] { $F^{(3)}_{m,n,k}$};
\end{tikzpicture}
 \end{center}
and the spectral radius of $F^{(3)}_{m,n,k}$ be $\rho(F^{(3)}_{m,n,k})$. Then, when $m,n,k\rightarrow \infty$,
$$\lim_{m,n,k\to\infty } \rho(F^{(3)}_{m,n,k})=2\sqrt[3]{2+\sqrt{5}}.$$
\end{lemma}
\begin{proof} We label this graph as follows
 
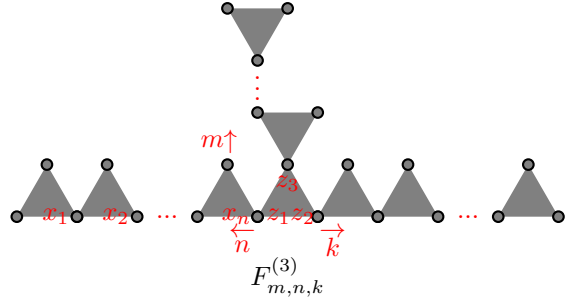
\begin{center}
\begin{tikzpicture}[thick, scale=0.8, bnode/.style={circle, draw,
    fill=black!50, inner sep=0pt, minimum width=4pt}, enode/.style={color=red}]
\path[fill=gray]  (4.5,0.866) -- (4,1.732) node [bnode] {} --
    (5,1.732)  node [bnode] {} --cycle;
\path[fill=gray]  (4,2.598) node [bnode] {}-- (3.5,3.464) node [bnode] {} --
    (4.5,3.464)  node [bnode] {} --cycle;

\foreach \x in {0,1,3,4,5,6,8}
    {
    \path[fill=gray]  (\x,0) node [bnode] {} -- (\x+0.5,0.866) node [bnode] {} --(\x+1,0)node [bnode] {}--cycle;
}
\draw (2.5,0) node [enode]{$...$};
\draw (7.5,0) node [enode]{$...$};

\draw (4,2.265) node [enode]{$\vdots$};

\draw (3.25,1.25) node [enode] {$m$};
\draw (3.55,1.25) node [enode] {$\uparrow$};
\draw (3.75,-0.25) node [enode] {$\leftarrow$};
\draw (3.75,-0.5) node [enode] {$n$};
\draw (5.25,-0.25) node [enode] {$\rightarrow$};
\draw (5.25,-0.5) node [enode] {$k$};
\draw (4.5, -1) node [color=black] { $F^{(3)}_{m,n,k}$};
\draw (0.65,0) node [enode] {$ x_{1}$};
\draw (1.65,0) node [enode] {$ x_{2}$};
\draw (3.65,0) node [enode] {$ x_{n}$};
\draw (4.35,0) node [enode] {$ z_{1}$};
\draw (4.75,0) node [enode] {$ z_{2}$};
\draw (4.5,0.56) node [enode] {$z_{3}$};
\end{tikzpicture}
 \end{center}
Let $\beta$ be a real number  in  $(0,\frac{1}{4})$, chosen later.
Set $z_{1}=1-x_{n}=1-f^{n-1}_\beta(\beta)$,  $z_{2}=1-f^{k-1}_\beta(\beta)$,
and $z_{3}=1-f^{m-1}_\beta(\beta)$.
Now let $\beta_{m,n,k}$ be the
solution of
$$z_1z_2z_3=\beta.$$
We get a $\beta_{n,m,k}$-normal labeling. Thus
$\rho(F^{(3)}_{m,n,k})=2\beta_{m,n,k}^{-1/3}$.
By the first item of Lemma \ref{infty},
for a fixed $\beta$, note that all $z_i$'s decreasingly approach to $\frac{1+\sqrt{1-4\beta}}{2}$.
We conclude that $\beta_{m,n,k}$ are decreasing functions of each $m$,
$n$, and $k$. The limit $\lim_{m,n,k\to\infty}\beta_{m,n,k}$ must exist and is the solution of
$$(\frac{1+\sqrt{1-4\beta}}{2})^{3}=\beta.$$
By simple calculus, we get this limit $\beta=\sqrt{5}-2$.
By Lemma \ref{l:main}, we get $\lim_{m,n,k\to\infty } \rho(F^{(3)}_{m,n,k})=2\sqrt[3]{2+\sqrt{5}}$.
\end{proof}
Taking $\rho'_r=(r-1)!\sqrt[r]{2+\sqrt{5}}$, we have the following lemma.
\begin{lemma}\label{pve}
For $r\geq 3$, let $H$ be an $r$-uniform hypergraph with spectral
radius $\rho(H)\leq \rho'_r$. If $H$ is not simple, then $H=C_2^{(r)}$
(i.e., the hypergraph consists of two edges sharing two common vertices).
\end{lemma}
\begin{proof}
In \cite{LM}, we have shown that
$\rho_r(C_2^{(r)})=(r-1)!\sqrt[r]{4}<\rho'_r$.

Since $H$ is not simple, $H$ contains two edges $F_1$ and $F_2$
sharing  $s$ vertices for some $s\geq 2$.

If $s\geq 3$, call the subgraph consisting of the two edges $F_1, F_2$
$C^{(r)}_{s+}$. Define a weighted incident matrix $B$ of $C^{(r)}_{s+}$ as follows: for any
vertex $v$ and edge $e$ (called the other edge $e'$),
$$B(v,e)=
\begin{cases}
 \frac{1}{2} & \mbox{ if } v\in e\cap e',\\
 1   & \mbox{ if }v\in e\setminus e',\\
0 &\mbox{ otherwise.}
\end{cases}
$$
It is easy to check that when $s\geq 3$ we have $(\frac{1}{2})^{s}<0.1251<\beta$,
so $C^{(r)}_{s+}$ is consistently $\beta$-supernormal and thus $\rho(H)\geq\rho(C^{(r)}_{s+})>\rho'_r$.
Contradiction!

Thus, $F_1$ and $F_2$ can only share $2$-common vertices. Since $H$ is
connected and $H\not=C_2^{(r)}$, there is a third edge $F_3$ having non-empty
intersection with $F_1\cup F_2$. Since identifying the vertices will
not change the sub-homomorphic type, we can only consider the two
 sub-homomorphic types: $C^{(r)}_{2+}$ and ${C'}^{(r)}_{2+}$.
Here both the hypergraphs $C^{(r)}_{2+}$ and ${C'}^{(r)}_{2+}$ consist
of three edges $F_1,F_2,F_3$ where $|F_1\cap F_2|=2$ and $|F_3\cap (F_1\cup F_2)|=1$.
The difference
is that in $C^{(r)}_{2+}$, $F_3\cap (F_1\cup F_2) \in F_1\cap F_2$
while in ${C'}^{(r)}_{2+}$,  $F_3\cap (F_1\cup F_2) \in F_1\Delta F_2$
the symmetric difference of $F_1$ and $F_2$. The below are the figures
of $C^{(3)}_{2+}$ and ${C'}^{(3)}_{2+}$.

 \begin{center}
\begin{tikzpicture}[thick, scale=0.8, bnode/.style={circle, draw,
    fill=black!50, inner sep=0pt, minimum width=4pt}, enode/.style={red}]

    \path[fill=gray]  (0,0) node [bnode] {} -- (0.866,0.5) node [bnode] {} --(0.866,-0.5)--cycle;

\draw (0.866,-0.5) node [bnode] {};
\draw (1.45,1.366) node [bnode] {};
\path[fill=black!30]  (0.866,0.5) node [bnode] {} -- (1.732,0) node [bnode] {} --(0.866,-0.5)--cycle;
\path[fill=gray]  (0.866,0.5) node [bnode] {} -- (0.2,1.366) node [bnode] {} --(1.45,1.366)--cycle;
\draw (1.732,-0.5) node [enode] {${C}^{(3)}_{2+}$};
\end{tikzpicture}
\hfil
\begin{tikzpicture}[thick, scale=0.8, bnode/.style={circle, draw,
    fill=black!50, inner sep=0pt, minimum width=4pt}, enode/.style={red}]

\foreach \x in {0,1.732}
    {
    \path[fill=gray]  (\x,0) node [bnode] {} -- (\x+0.866,0.5) node [bnode] {} --(\x+0.866,-0.5)--cycle;
}

\draw (0.866,-0.5) node [bnode] {};
\draw (2.598,-0.5) node [bnode] {};
\path[fill=black!30]  (0.866,0.5) node [bnode] {} -- (1.732,0) node [bnode] {} --(0.866,-0.5)--cycle;
\draw (1.732,-0.5) node [enode] {${C'}^{(3)}_{2+}$};
\end{tikzpicture}

\end{center}
To draw the contradiction, it is sufficient to show
$\rho_r({C}^{(r)}_{2+})>\rho_r'$ and $\rho_r({C'}^{(r)}_{2+})>\rho_r'$
(this implies $\rho(H)>\rho_r'$ by Lemma \ref{subhomo}).
Observe that ${C}^{(r)}_{2+}$ is extended from ${C}^{(3)}_{2+}$
and ${C'}^{(r)}_{2+}$ is extended from ${C'}^{(3)}_{2+}$. We only need
to show that both ${C}^{(3)}_{2+}$ and ${C'}^{(3)}_{2+}$ are
consistently strict $\beta$-supernormal. We label the two hypergraphs
as follows:
 
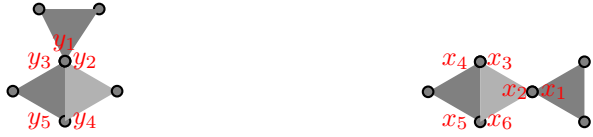
\begin{center}
\begin{tikzpicture}[thick, scale=0.8, bnode/.style={circle, draw,
    fill=black!50, inner sep=0pt, minimum width=4pt}, enode/.style={red}]

\foreach \x in {0}
    {
    \path[fill=gray]  (\x,0) node [bnode] {} -- (\x+0.866,0.5) node [bnode] {} --(\x+0.866,-0.5)--cycle;
}

\draw (0.866,-0.5) node [bnode] {};
\draw (1.435,1.366) node [bnode] {};
\path[fill=black!30]  (0.866,0.5) node [bnode] {} -- (1.732,0) node [bnode] {} --(0.866,-0.5)--cycle;
\path[fill=gray]  (0.866,0.5) node [bnode] {} -- (0.435,1.366) node [bnode] {} --(1.435,1.366)--cycle;
\draw (0.866,0.8) node [enode] {$y_{1}$};
\draw (1.2,0.5) node [enode] {$y_{2}$};
\draw (0.450,0.5) node [enode] {$y_{3}$};
\draw (1.2,-0.5) node [enode] {$y_{4}$};
\draw (0.45,-0.5) node [enode] {$y_{5}$};
\end{tikzpicture}
\hfil
\begin{tikzpicture}[thick, scale=0.8, bnode/.style={circle, draw,
    fill=black!50, inner sep=0pt, minimum width=4pt}, enode/.style={red}]

\foreach \x in {0,1.732}
    {
    \path[fill=gray]  (\x,0) node [bnode] {} -- (\x+0.866,0.5) node [bnode] {} --(\x+0.866,-0.5)--cycle;
}

\draw (0.866,-0.5) node [bnode] {};
\draw (2.598,-0.5) node [bnode] {};
\path[fill=black!30]  (0.866,0.5) node [bnode] {} -- (1.732,0) node [bnode] {} --(0.866,-0.5)--cycle;

\draw (2.1,0) node [enode] {$x_{1}$};
\draw (1.45,0) node [enode] {$x_{2}$};
\draw (1.2,0.5) node [enode] {$x_{3}$};
\draw (0.450,0.5) node [enode] {$x_{4}$};
\draw (1.2,-0.5) node [enode] {$x_{6}$};
\draw (0.45,-0.5) node [enode] {$x_{5}$};
\end{tikzpicture}
\end{center}

In $C^{(3)}_{2+}$, we set the labels $y_{1}=\beta$,
$y_{2}=y_{3}=\frac{1-\beta}{2}$,
and $y_{4}=y_{5}=\frac{2\beta}{1-\beta}$.
Since $y_{4}+y_{5}\approx 1.2361>1$,
this is a consistently $\beta$-supernormal labelling.

In ${C'}^{(3)}_{2+}$, we set $x_{1}=\beta$, $x_{2}=1-\beta$,
$x_{3}=x_{6}=\sqrt{\frac{\beta}{1-\beta}}$,
 and $x_{4}=x_{5}=\sqrt{\beta}$. Since $x_{3}+x_{4}=x_5+x_6\approx
1.0418>1$, this is a consistently $\beta$-supernormal labelling.

\end{proof}
\section{Proof of Theorem \ref{t1}}

\begin{proof} It suffices to consider irreducible hypergraphs. Assume that $H$ is an irreducible $3$-uniform
hypergraph with $\rho(H)\leq 2 \sqrt[3]{2+\sqrt{5}}$. We need to show that $H$
has certain forbidden structures. The idea is to show these forbidden subgraphs
have some (consistently, if not a hypertree)
$(\sqrt{5}-2)$-supernormal labelings. To simplify our notation, we
write $\beta=\sqrt{5}-2$ in this proof.
By Lemma \ref{pve}, when $r=3$, we only need to consider $H$ is simple.\\
{\bf Case 1.}  If $\exists~ v\in V(H)$, such that $d_{v}\geq 5$, then $H$ contains $S^{(3)}_{5}$ that has been labeled as follows.
  
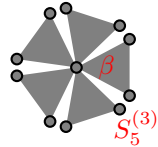
\begin{center}
 \begin{tikzpicture}[thick, scale=0.8, bnode/.style={circle, draw,
    fill=black!50, inner sep=0pt, minimum width=4pt}, enode/.style={red}]
\foreach \x in {0,72,...,288}
    {
    \path[fill=gray]  (0,0) -- (\x-28:1) node [bnode] {} --(\x+28:1) node
    [bnode] {} --cycle;
}
\draw (0,0) node  [bnode] {};
\draw (1,-1) node [enode] {$ S_5^{(3)}$};
\draw (0.5,0) node [enode] {$ \beta$};
\end{tikzpicture}
\end{center}
By the symmetry, we only label one branching. We can check $5\beta\approx 1.1803>1$, so, by Lemma \ref{subgraph} and Lemma \ref{supernormal},
we get $\rho(H)>\rho'_{3}$. Thus
we can assume that every vertex in $H$ has degree at most 4.\\
If $\exists ~ v\in V(H)$, such that $d_{v}=4$, and $H$ contains graph $S^{(3)}_{4+}$ that has been labeled as follows,
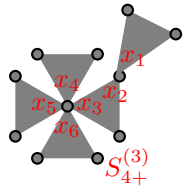
\begin{center}
\begin{tikzpicture}[thick, scale=0.8, bnode/.style={circle, draw,
    fill=black!50, inner sep=0pt, minimum width=4pt}, enode/.style={red}]
\foreach \x in {0,90,...,270}
    {
    \path[fill=gray]  (0,0) -- (\x-30:1) node [bnode] {} --(\x+30:1) node
    [bnode] {} --cycle;
}
\draw (0,0) node  [bnode] {};
\draw (1,-1) node [enode] {$S^{(3)}_{4+}$};
\path[fill=gray]  (0.8,0.5) -- (1.8,1.2) node [bnode] {} --(1.0,1.6) node
    [bnode] {} --cycle;
\draw (1.1,0.8) node [enode] {$ x_{1}$};
\draw (0.8,0.2) node [enode] {$ x_{2}$};
\draw (0.38,0) node [enode] {$ x_{3}$};
\draw (0,0.38) node [enode] {$ x_{4}$};
\draw (-0.38,0) node [enode] {$ x_{5}$};
\draw (0,-0.38) node [enode] {$ x_{6}$};
\end{tikzpicture}
\end{center}
where $x_{1}=\beta$, $x_{2}=1-\beta$, $x_{3}=\frac{\beta}{1-\beta}$, $x_{4}=x_{5}=x_{6}=\beta$.
We can check that $x_{3}+x_{4}+x_{5}+x_{6}\approx  1.0172>1$, so, by Lemma \ref{subgraph} and Lemma \ref{supernormal}, we get $\rho(H)>\rho(S^{(3)}_{4+})>\rho'_{3}$.
Thus, since $\rho(S^{(3)}_{4})=\rho_{3}$ and $\rho(S^{(3)}_{4+})>\rho'_{3}$,
so if $H$ is irreducible, we can assume that every vertex in $H$ has degree at most 3.\\
{\bf Case 2.} The hypergraph $H$ contains a cycle, saying $C^{(3)}_n$.
Since $\rho(C^{(3)}_n)=2\sqrt[3]{4}$ (see \cite{LM}), we may assume
$H$ contains at least one edge $F$ not on the cycle $C^{(3)}_n$
(but attached to $C^{(3)}_n$). First we prove that $F$ can be
only attached to the cycle through
a branching edge, not a branching vertex,
otherwise,  $H$ contains a sub-homomorphic type
  ${C}_{n+}^{(3)}$ shown as follows:
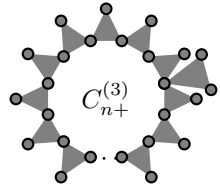
\begin{center}
\begin{tikzpicture}[thick, scale=0.4, bnode/.style={circle, draw,
    fill=black!50, inner sep=0pt, minimum width=4pt}, enode/.style={color=red}]
\foreach \x in {-60,-30,...,240}
    {
    \path[fill=gray]  (\x-15:2) node [bnode] {} -- (\x:3) node [bnode] {} --(\x+15:2)--cycle;
}
\draw (270:2) node  [color=black] {$\cdots$};
\draw (255:2) node [bnode] {};
\draw (0,0) node [color=black] {${C}^{(3)}_{n+}$};
\path[fill=gray]  (15:2) node [bnode] {} -- (25:3.5) node [bnode] {} --(5:3.5)node [bnode] {}--cycle;
\end{tikzpicture}
\end{center}
This graph is reducible and can be extended from the following 2-graph $C^{(2)}_{n+}$:
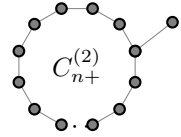
\begin{center}
\begin{tikzpicture}[thick, scale=0.4, bnode/.style={circle, draw,
    fill=black!50, inner sep=0pt, minimum width=4pt}, enode/.style={color=red}]
\foreach \x in {-60,-30,...,240}
    {
    \path[fill=gray]  (\x-15:2) node [bnode] {} --(\x+15:2)--cycle;
}
\draw (270:2) node  [color=black] {$\cdots$};
\draw (255:2) node [bnode] {};
\draw (0,0) node [color=black] {$C^{(2)}_{n+}$};
\path[fill=gray]  (15:2) node [bnode] {} -- (25:3.5) node [bnode] {}--cycle;
\end{tikzpicture}
\end{center}
The graph $C^{(2)}_{n+}$ is not in the list of Brouwer and Neumaier
(see Page 1). Thus, $\rho(C^{(2)}_{n+})>\sqrt{2 + \sqrt{5}}$. Applying
Corollary \ref{extend}, we get $\rho(C^{(3)}_{n+})>2\sqrt[3]{2 + \sqrt{5}}$.
Contradiction!

Thus, $F$ must be attached to the cycle through a branching edge.
Considering that we walk away from the cycle through this edge $F$,
we have the following subcases.
\begin{enumerate}

\item Eventually, the path leaving at $F$ reaches a branching vertex.
In this subcase, $H$ contains the following sub-homomorphic type ${C'}^{(3)}_{n+}$:
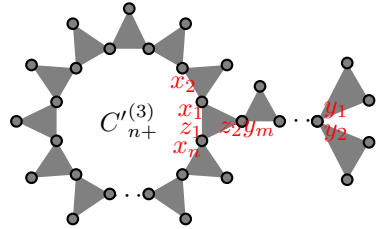
\begin{center}
\begin{tikzpicture}[thick, scale=0.5, bnode/.style={circle, draw,
    fill=black!50, inner sep=0pt, minimum width=4pt}, enode/.style={color=red}]
\foreach \x in {-60,-30,...,240}
    {
    \path[fill=gray]  (\x-15:2) node [bnode] {} -- (\x:3) node [bnode] {} --(\x+15:2)--cycle;
}
\draw (270:2) node  [color=black] {$\cdots$};
\draw (255:2) node [bnode] {};
\draw (0,0) node [color=black] {${C'}^{(3)}_{n+}$};
\draw (1.55,-0.75) node [enode] {$x_{n}$};
\draw (1.65,0.25) node [enode] {$x_{1}$};
\draw (1.65,-0.25) node [enode] {$z_{1}$};
\draw (1.45,0.99) node [enode] {$x_{2}$};
\draw (2.7,-0.2) node [enode] {$z_{2}$};
\draw (4.5,0) node [color=black] {$\cdots$};
\path[fill=gray]  (0:5) node [bnode] {} -- (-15:6) node [bnode] {} --(-5:6.4)node [bnode] {}--cycle;
\path[fill=gray]  (0:5) node [bnode] {} -- (15:6) node [bnode] {} --(5:6.4)node [bnode] {}--cycle;
\draw (5.5,0.3) node [enode] {$y_{1}$};
\draw (5.5,-0.3) node [enode] {$y_{2}$};
\path[fill=gray]  (0:3) node [bnode] {} -- (15:3.6) node [bnode] {} --(0:4)node [bnode] {}--cycle;
\draw (3.45,-0.2) node [enode] {$y_{m}$};
\end{tikzpicture}
\end{center}
By Lemma \ref{circle}, there exists a
 $x_{1}\in (\frac{1-\sqrt{1-4\beta}}{2}, \frac{1+\sqrt{1-4\beta}}{2})$
satisfying $f_{\beta}^{n}(x_{1})=1-x_{1}$.
Now $x_{n}=f_{\beta}^{n}(x_{1})=1-x_{1}$. (This symmetry guarantees
the labeling is consistent.)
So,
$z_{1}=1-x_{n}=x_{1}$. We set $y_{1}=y_{2}=\beta$,
$y_i=f_\beta^{i-2}(2\beta)$ for $3\leq i\leq m$.
Since $\frac{1-\sqrt{1-4\beta}}{2}<2\beta<\frac{1+\sqrt{1-4\beta}}{2}$,
 by Lemma \ref{infty}, we get that $y_{i}$ is decreasing
and the limit goes to $\frac{1-\sqrt{1-4\beta}}{2}$.
In particular, $y_m>\frac{1-\sqrt{1-4\beta}}{2}$. It implies
 $z_{2}=1-y_{m}<\frac{1+\sqrt{1-4\beta}}{2}$.
Therefore, we have
$$x_{1}\cdot z_{1}\cdot
z_{2}<(\frac{1+\sqrt{1-4\beta}}{2})^{3}=\beta.$$
Thus, ${C'}^{(3)}_{n+}$ is consistently $\beta$-supernormal.
So, we have $\rho(H)\geq \rho({C'}^{(3)}_{n+})>\rho'_{3}$. Contradiction!

\item  Eventually, the path leaving at $F$ reaches a branching edge.
In this subcase, $H$ contains the following sub-homomorphic type ${C''}^{(3)}_{n+}$:
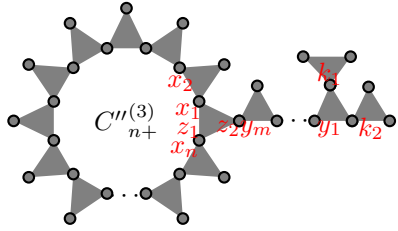
\begin{center}
\begin{tikzpicture}[thick, scale=0.5, bnode/.style={circle, draw,
    fill=black!50, inner sep=0pt, minimum width=4pt}, enode/.style={color=red}]
\foreach \x in {-60,-30,...,240}
    {
    \path[fill=gray]  (\x-15:2) node [bnode] {} -- (\x:3) node [bnode] {} --(\x+15:2)--cycle;
}
\draw (270:2) node  [color=black] {$\cdots$};
\draw (255:2) node [bnode] {};
\draw (0,0) node [color=black] {${C''}^{(3)}_{n+}$};
\draw (1.55,-0.75) node [enode] {$x_{n}$};
\draw (1.65,0.25) node [enode] {$x_{1}$};
\draw (1.65,-0.25) node [enode] {$z_{1}$};
\draw (1.45,0.99) node [enode] {$x_{2}$};
\draw (2.7,-0.2) node [enode] {$z_{2}$};
\draw (4.5,0) node [color=black] {$\cdots$};
\path[fill=gray]  (0:5) node [bnode] {} -- (10:5.5) node [bnode] {} --(0:6)node [bnode] {}--cycle;
\path[fill=gray]  (0:6) node [bnode] {} -- (8:6.5) node [bnode] {} --(0:7)node [bnode] {}--cycle;
\path[fill=gray]  (10:5.5) node [bnode] {} -- (20:5) node [bnode] {} --(16:6.2)node [bnode] {}--cycle;
\draw (5.4,1.25) node [enode] {$k_{1}$};
\draw (5.4,-0.2) node [enode] {$y_{1}$};
\draw (6.5,-0.2) node [enode] {$k_{2}$};
\path[fill=gray]  (0:3) node [bnode] {} -- (15:3.6) node [bnode] {} --(0:4)node [bnode] {}--cycle;
\draw (3.45,-0.2) node [enode] {$y_{m}$};
\end{tikzpicture}
\end{center}

This is very similar to the previous subcase.
By Lemma \ref{circle}, there exists a
 $x_{1}\in (\frac{1-\sqrt{1-4\beta}}{2}, \frac{1+\sqrt{1-4\beta}}{2})$
satisfying $f_{\beta}^{n}(x_{1})=1-x_{1}$.
Now $x_{n}=f_{\beta}^{n}(x_{1})=1-x_{1}$. (This symmetry guarantees
the labeling is consistent.)
So,
$z_{1}=1-x_{n}=x_{1}$. We set $k_{1}=k_{2}=\beta$,
$y_1=\frac{\beta}{(1-\beta)^2}$, and
$y_i=f_\beta^{i-1}(y_1)$ for $2\leq i\leq m$.
Since $\frac{1-\sqrt{1-4\beta}}{2}<\frac{\beta}{(1-\beta)^2} <\frac{1+\sqrt{1-4\beta}}{2}$,
 by Lemma \ref{infty}, we get that $y_{i}$ is decreasing
and the limit goes to $\frac{1-\sqrt{1-4\beta}}{2}$.
In particular, $y_m>\frac{1-\sqrt{1-4\beta}}{2}$. It implies
 $z_{2}=1-y_{m}<\frac{1+\sqrt{1-4\beta}}{2}$.
Therefore, we have
$$x_{1}\cdot z_{1}\cdot
z_{2}<(\frac{1+\sqrt{1-4\beta}}{2})^{3}=\beta.$$
Thus, ${C''}^{(3)}_{n+}$ is consistently $\beta$-supernormal.
So, we have $\rho(H)\geq \rho({C''}^{(3)}_{n+})>\rho'_{3}$. Contradiction!

\item  Eventually, the path leaving at $F$ returns to the cycle.
In this subcase, $H$ contains subgraph $\Theta(m_1,m_2,m_3)$,
which can be obtained by connecting three pairs of vertices between
two branching edges using three paths of lengths $m_1$, $m_2$, and
$m_3$ respectively.
\begin{center}
\begin{tikzpicture}[thick, scale=1, bnode/.style={circle, draw,
    fill=black!50, inner sep=0pt, minimum width=4pt}, enode/.style={color=red}]

\path[fill=gray] (0,0) node [bnode] {}-- (0,1) node [bnode] {}-- (0.866,0.5) node [bnode] {}-- cycle;
\path[fill=gray] (4,0) node [bnode] {} -- (4,1) node [bnode] {}--
(4-0.866,0.5) node [bnode] {} -- cycle;
\draw[dotted] (0,0) -- node {$m_1$} (4,0);
\draw[dotted] (0,1) -- node {$m_3$} (4,1);
\draw[dotted] (0.866,0.5) -- node {$m_2$} (4-0.866,0.5);
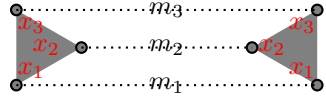
\draw (2,-1) node {$\Theta(m_1,m_2,m_3)$ and its labeling};
\draw (0.2,0.2) node [enode] {$x_{1}$};
\draw (3.8,0.2) node [enode] {$x_{1}$};
\draw (0.4,0.5) node [enode] {$x_{2}$};
\draw (3.4,0.5) node [enode] {$x_{2}$};
\draw (0.2,0.8) node [enode] {$x_{3}$};
\draw (3.8,0.8) node [enode] {$x_{3}$};
\end{tikzpicture}
\end{center}

By Lemma \ref{circle}, for $i=1,2,3$,  there exists a
 $x_{i}\in (\frac{1-\sqrt{1-4\beta}}{2}, \frac{1+\sqrt{1-4\beta}}{2})$
satisfying $f_{\beta}^{m_i}(x_{i})=1-x_{i}$.
We label $x_1$, $x_2$, and $x_3$ on the $\Theta(m_1,m_2,m_3)$ and
extend these labels on path $P_{m_i}$ naturally. The definition of $x_i$
makes the labelings on $P_{m_i}$ symmetric and this symmetry guarantees
the labeling is consistent. Note
$$x_1x_2x_3<\left(\frac{1+\sqrt{1-4\beta}}{2}\right)^3=\beta.$$
This is consistently $\beta$-supernormal and this implies
$$\rho(H)\geq \rho(\Theta(m_1,m_2,m_3))>\rho'_3.$$
Contradiction!

\item This is the remaining subcase: $H$ contains a cycle $C$ with several
  path attached to $C$. So $H$ is a closed quipu as stated in the theorem.
\end{enumerate}
{\bf Case 3.} We assume that $H$ is a hypertree, and let the following partial hypergraphs denote $H_{1}^{(3)}$ and $H_{2}^{(3)}$ that correspond to the branching vertex and the branching edge structure respectively.
 
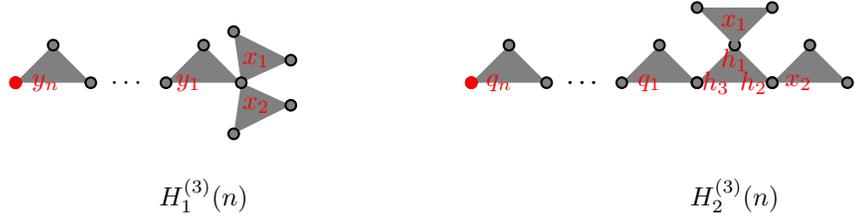
\begin{center}
\begin{tikzpicture}[thick, scale=1.0, bnode/.style={circle, draw,
    fill=black!50, inner sep=0pt, minimum width=4pt}, enode/.style={red}]
  \path[fill=gray]  (1,0) node [bnode] {} -- (1.5,0.5) node [bnode] {} --(2,0) --cycle;
\path[fill=gray]  (2,0)  -- (2.66,-0.30) node [bnode] {} -- (1.90,-0.68) node [bnode]{} --cycle;
\path[fill=gray]  (2,0)  -- (2.66,0.30) node [bnode] {} -- (1.90,0.68) node [bnode]{} --cycle;
\draw (1.5, -1.5) node [color=black] { $H^{(3)}_{1}(n)$};
\draw (2,0) node  [bnode] {};
\draw (0.5,0) node [color=black] {$\cdots$};
\path[fill=gray]  (-1,0) node [bnode] {} -- (-0.5,0.5) node [bnode] {} --(0,0)node [bnode] {} --cycle;
\draw (-1,0) node [bnode][color=red] {};
\draw (-0.6,0) node [color=red] {$y_{n}$};
\draw (2.2,0.3) node [color=red] {$x_{1}$};
\draw (2.2,-0.3) node [color=red] {$x_{2}$};
\draw (1.3,0) node [color=red] {$y_{1}$};
\end{tikzpicture}
\hfil
\begin{tikzpicture}[thick, scale=1.0, bnode/.style={circle, draw,
    fill=black!50, inner sep=0pt, minimum width=4pt}, enode/.style={red}]
    \foreach \x in {1,3,4,5}
    {
    \path[fill=gray]  (\x,0) node [bnode] {} -- (\x+0.5,0.5) node [bnode] {} --(\x+1,0)node [bnode] {}--cycle;
}

\draw (2.5,0) node [color=black] {$\cdots$};
\path[fill=gray]  (4.5,0.5)  -- (4,1) node [bnode] {} -- (5,1) node [bnode]{} --cycle;

\draw (2.5,0) node [color=black] {$\cdots$};
\draw (4.5, -1.5) node [color=black] { $H^{(3)}_{2}(n)$ };
\draw (1,0) node [bnode][color=red] {};
 \draw (1.37,0) node [color=red] {$q_{n}$};
\draw (4.5,0.8) node [color=red] {$x_{1}$};
\draw (5.35,0) node [color=red] {$x_{2}$};
\draw (3.37,0) node [color=red] {$q_{1}$};
\draw (4.5,0.3) node [color=red] {$h_{1}$};
\draw (4.25,0) node [color=red] {$h_{3}$};
\draw (4.75,0) node [color=red] {$h_{2}$};
\end{tikzpicture}
\end{center}

In graph $H^{(3)}_{1}(n)$, we set $x_{1}=x_{2}=\beta$, $y_{1}=\frac{\beta}{1-2\beta}=f_{\beta}(2\beta)$, $y_{n}=f^{n}_{\beta}(2\beta)$.
Since $2\beta\in(\frac{1-\sqrt{1-4\beta}}{2}, \frac{1+\sqrt{1-4\beta}}{2})$, by Lemma \ref{infty},
 we get that $y_{n}=f^n_\beta>\frac{1-\sqrt{1-4\beta}}{2}$.

In graph $H^{(3)}_{2}(n)$, we set $x_{1}=x_{2}=\beta$,
$h_{1}=h_{2}=1-\beta$,
$h_{3}=\frac{\beta}{(1-\beta)^2}$.
We can check that
$h_{3}\in (\frac{1-\sqrt{1-4\beta}}{2}, \frac{1+\sqrt{1-4\beta}}{2})$.
Since $q_{n}=f^n_{\beta}(h_{3})$, and thus by Lemma \ref{infty},
we get $q_{n}>\frac{1-\sqrt{1-4\beta}}{2}$.\\

To show $H$ must be an open quipu as stated in the theorem, we need
exclude the following structures.
 First, suppose that there is a branching vertex in the middle of $H$,
and $H$ contains the following subgraph,
  
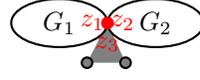
\begin{center}
\begin{tikzpicture}[thick, scale=0.8, bnode/.style={circle, draw,
    fill=black!50, inner sep=0pt, minimum width=4pt}, enode/.style={red}]
 %    \foreach \x in {1,2}
%     {
%     \path[fill=gray]  (\x,0) node [bnode] {} -- (\x+0.5,0.5) node [bnode] {} --(\x+1,0)node [bnode] {}-- (\x+0.5,-0.5) node [bnode]{} --cycle;
% }
\path[fill=gray]  (2,0)  -- (1.66,-0.66) node [bnode] {} --(2.3,-0.66)node [bnode] {} --cycle;
\draw (1.2,0) ellipse (0.8 and 0.4);
\draw (1.2,0) node [color=black] {$G_{1}$};
\draw (2.8,0) ellipse (0.8 and 0.4);
\draw (2.8,0) node [color=black] {$G_{2}$};
  \draw (1.75,0) node [color=red] {$z_{1}$};
  \draw (2.25,0) node [color=red] {$z_{2}$};
  \draw (2,-0.35) node [color=red] {$z_{3}$};
  \draw (2,0) node [bnode][color=red] {};
\end{tikzpicture}
\end{center}
where $ G_{1}$ and $ G_{2}$ are chosen from $H^{(3)}_{1}(n)$ and $H^{(3)}_{2}(n)$ (for some $n\geq 0$) and pieces are glued through red nodes.
 We can get
$$z_1+z_2+\beta>\frac{1-\sqrt{1-4\beta}}{2}
+\frac{1-\sqrt{1-4\beta}}{2} +\beta=1.$$
This is a supernormal labeling of this subgraph. Thus,
$\rho(H)>\rho_3'$. Contradiction!

\noindent If $H$ contains one branching edge, whose all three branches
are not paths, then $H$ contains the following subgraph.
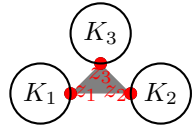
\begin{center}
\begin{tikzpicture}[thick, scale=0.8, bnode/.style={circle, draw,
    fill=black!50, inner sep=0pt, minimum width=4pt}, enode/.style={red}]
    \foreach \x in {1}
    {
    \path[fill=gray]  (\x,0) node [bnode] {} -- (\x+0.5,0.5) node [bnode] {} --(\x+1,0)node [bnode] {} --cycle;
}
\draw (1,0) node [bnode][color=red] {};
\draw (2,0) node [bnode][color=red] {};
\draw (1.5,0.5) node [bnode][color=red] {};
\draw (1.5,0.3) node [color=red] {$z_{3}$};
\draw (1.25,0) node [color=red] {$z_{1}$};
\draw (1.75,0) node [color=red] {$z_{2}$};
\draw (0.5,0) ellipse (0.5);
\draw (2.5,0) ellipse (0.5);
\draw (1.5,1.0) ellipse (0.5);
\draw (0.5,0) node [color=black] {$K_{1}$};
\draw (2.5,0) node [color=black] {$K_{2}$};
\draw (1.5,1.0) node [color=black] {$K_{3}$};

\end{tikzpicture}
\end{center}
where $K_{1}$, $K_{2}$ and $K_{3}$ are chosen from $H^{(3)}_{1}(n)$ and $H^{(3)}_{2}(n)$ (for some $n\geq 0$) and pieces are glued through red nodes.
Similar to the previous case, for $i=1,2,3$, by Lemma \ref{infty}, we can get
$z_i<\frac{1+\sqrt{1-4\beta}}{2}$.
Thus, $$z_{1}\cdot z_{2}\cdot
z_{3}<(\frac{1+\sqrt{1-4\beta}}{2})^{3}=\beta.$$
This is a supernormal labeling of this subgraph.
So, we have $\rho(H)>\rho'_{3}$. Contradiction!
Therefore $H$ must be an open quipu as stated in the theorem.
\end{proof}
\section{Proof of Theorem \ref{t2}}

\begin{proof}
 Let $H$ be an irreducible $4$-uniform hypergraph with $\rho(H)\leq
 \rho'_4=6\sqrt[4]{2+\sqrt{5}}$. If $H$ is not simple, then it must
be $C_2^{(4)}$ by Lemma \ref{pve}.  Now we consider $H$ is simple.

\noindent
 {\bf Case 1.} $H$ contains a cycle $C$. Since $H$ is irreducible, it
 also has an edge $F$ which contains no leaf vertex.
 We consider the following two subcases.
  \begin{enumerate}
 \item The edge $F$ is on the cycle $C$. The $H$ contains the
   following sub-isomorphic type:
 
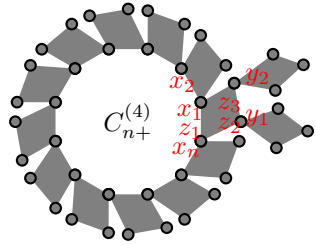
\begin{center}
\begin{tikzpicture}[thick, scale=0.5, bnode/.style={circle, draw,
    fill=black!50, inner sep=0pt, minimum width=4pt}, enode/.style={color=red}]
\foreach \x in {-60,-30,0,30,60,90,120,150,180,210,240,270}
    {
    \path[fill=gray]  (\x-15:2) node [bnode] {} -- (\x:3) node [bnode] {} --(\x+20:3) node [bnode] {}--(\x+15:2) node [bnode] {}--cycle;
}
\draw (255:2) node [bnode] {};
\path[fill=gray]  (20:3) node [bnode] {} -- (25:4.25) node [bnode] {}-- (18:4.9) node [bnode] {}--(12:4)node [bnode] {}--cycle;
\path[fill=gray]  (0:3) node [bnode] {} -- (5:4.0) node [bnode] {}-- (-3:4.778) node [bnode] {} --(-10:4)node [bnode] {}--cycle;
\draw (0,0) node [color=black] {$C^{(4)}_{n+}$};
\draw (1.65,0.25) node [enode] {$x_{1}$};
\draw (1.65,-0.25) node [enode] {$z_{1}$};
\draw (1.45,0.99) node [enode] {$x_{2}$};
\draw (2.7,-0.1) node [enode] {$z_{2}$};
\draw (1.55,-0.75) node [enode] {$x_{n}$};
\draw (2.7,0.45) node [enode] {$z_{3}$};
\draw (3.45,1.2) node [enode] {$y_{2}$};
\draw (3.45,0.1) node [enode] {$y_{1}$};
\end{tikzpicture}
\end{center}
By Lemma \ref{circle}, there exists a
 $x_{1}\in (\frac{1-\sqrt{1-4\beta}}{2}, \frac{1+\sqrt{1-4\beta}}{2})$
satisfying $f_{\beta}^{n}(x_{1})=1-x_{1}$.
Now $x_{n}=f_{\beta}^{n}(x_{1})=1-x_{1}$. (This symmetry guarantees
the labeling is consistent.)
So, $z_{1}=1-x_{n}=x_{1}$. We set $y_{1}=y_{2}=\beta$, $z_{2}=z_{3}=1-\beta$,
and we can check that $x_{1}\cdot z_{1}\cdot z_{2}\cdot z_{3}<(\frac{1+\sqrt{1-4\beta}}{2})^{2}\cdot (1-\beta)^{2}\approx 0.2229<\beta$,
 and thus $C^{(4)}_{n+}$ is $\beta$-supernormal. So we have $\rho(H)\geq\rho(C^{(4)}_{n+})>\rho'_r$.

\item If $F$ is not on $C$, there is a path connecting $F$ to
  $C$. Thus, $H$ has the following sub-homomorphic type:
 
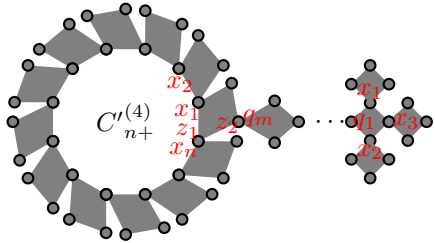
\begin{center}
\begin{tikzpicture}[thick, scale=0.5, bnode/.style={circle, draw,
    fill=black!50, inner sep=0pt, minimum width=4pt}, enode/.style={color=red}]
\foreach \x in {-60,-30,0,30,60,90,120,150,180,210,240,270}
    {
    \path[fill=gray]  (\x-15:2) node [bnode] {} -- (\x:3) node [bnode] {} --(\x+20:3) node [bnode] {}--(\x+15:2) node [bnode] {}--cycle;
}
\draw (255:2) node [bnode] {};
\path[fill=gray]  (0:3) node [bnode] {} -- (8:4.0) node [bnode] {}-- (0:4.6) node [bnode] {} --(-8:4)node [bnode] {}--cycle;

\draw (0,0) node [color=black] {${C'}^{(4)}_{n+}$};
\draw (1.65,0.25) node [enode] {$x_{1}$};
\draw (1.65,-0.25) node [enode] {$z_{1}$};
\draw (1.45,0.99) node [enode] {$x_{2}$};
\draw (2.7,-0.1) node [enode] {$z_{2}$};
\draw (1.55,-0.75) node [enode] {$x_{n}$};
\draw (3.54,0.1) node [enode] {$q_{m}$};
\draw (5.5,0) node [color=black] {$\cdots$};
\foreach \x in {6,7}
    {
    \path[fill=gray]  (\x,0) node [bnode] {} -- (\x+0.5,0.5) node [bnode] {} --(\x+1,0)node [bnode] {}-- (\x+0.5,-0.5) node [bnode]{} --cycle;
}

\path[fill=gray]  (6.5,0.5)  -- (6,1) node [bnode] {} --(6.5,1.5)node
[bnode] {}-- (7,1) node [bnode]{} --cycle;
\path[fill=gray]  (6.5,-0.5)  -- (6,-1) node [bnode] {} --(6.5,-1.5)node
[bnode] {}-- (7,-1) node [bnode]{} --cycle;
\draw (6.5,0.8) node [color=red] {$x_{1}$};
\draw (6.5,-0.8) node [color=red] {$x_{2}$};
\draw (6.35,0) node [color=red] {$q_{1}$};
\draw (7.45,0) node [color=red] {$x_{3}$};
\end{tikzpicture}
\end{center}
As above, there exists a
 $x_{1}\in (\frac{1-\sqrt{1-4\beta}}{2}, \frac{1+\sqrt{1-4\beta}}{2})$
and $z_{1}=x_{1}$. We set $x_{1}=x_{2}=x_{3}=\beta$, $q_{1}=\frac{\beta}{(1-\beta)^{3}}$,
and we can check $q_{1}\in (\frac{1-\sqrt{1-4\beta}}{2}, \frac{1+\sqrt{1-4\beta}}{2})$.
 We set $q_{m}=f_{\beta}^{m-1}(q_{1})$, and thus by Lemma \ref{infty},
we get $q_{m}$ decreases with $m$,  and when $m\rightarrow \infty$, we get $q_{m}>\frac{1-\sqrt{1-4\beta}}{2}$.
So $z_{2}=1-q_{m}<\frac{1+\sqrt{1-4\beta}}{2}$. We can check that $x_{1}\cdot z_{1}\cdot z_{2}<(\frac{1+\sqrt{1-4\beta}}{2})^{3}=\beta$,
 and thus ${C'}^{(4)}_{n+}$ is $\beta$-supernormal. So we have
 $\rho(H)\geq \rho({C'}^{(4)}_{n+})>\rho'_r$.
\end{enumerate}
{\bf Case 2.} $H$ is a hypertree but not a $4$-dagger.  To get the open quipu structures,
we need forbid certain subhypergraphs.

The following partial hypergraphs  $H^{(4)}_{1}(n)$ and
$H^{(4)}_{2}(n,j)$ (for $j=0,1,2,3$) correspond to the branching
vertex and the branching edge structure respectively.

 \begin{center}
\begin{tikzpicture}[thick, scale=0.8, bnode/.style={circle, draw,
    fill=black!50, inner sep=0pt, minimum width=4pt}, enode/.style={red}]
  \path[fill=gray]  (1,0) node [bnode] {} -- (1.5,0.5) node [bnode] {} --(2,0)-- (1.5,-0.5) node [bnode]{} --cycle;

\path[fill=gray]  (2,0)  -- (2.66,-0.30) node [bnode] {} --(2.56,-0.86)node
[bnode] {}-- (1.90,-0.68) node [bnode]{} --cycle;
\path[fill=gray]  (2,0)  -- (2.66,0.30) node [bnode] {} --(2.56,0.86)node
[bnode] {}-- (1.90,0.68) node [bnode]{} --cycle;
\draw (1.5, -1.5) node [color=black] { $H^{(4)}_{1}(n)$};
\draw (2,0) node  [bnode] {};
\draw (0.5,0) node [color=black] {$\cdots$};
\path[fill=gray]  (-1,0) node [bnode] {} -- (-0.5,0.5) node [bnode] {} --(0,0)node [bnode] {}-- (-0.5,-0.5) node [bnode]{} --cycle;
\draw (1.35,0.65) node [enode] {$\leftarrow$};
\draw (1.5,0.95) node [enode] {$n$};
\draw (-1,0) node [bnode][color=red] {};
\draw (-0.6,0) node [color=red] {$y_{n}$};
\draw (2.2,0.3) node [color=red] {$x_{1}$};
\draw (2.2,-0.3) node [color=red] {$x_{2}$};
\draw (1.3,0) node [color=red] {$y_{1}$};
\end{tikzpicture}
\hfil
\begin{tikzpicture}[thick, scale=0.8, bnode/.style={circle, draw,
    fill=black!50, inner sep=0pt, minimum width=4pt}, enode/.style={red}]
    \foreach \x in {1,3,4,5,7}
    {
    \path[fill=gray]  (\x,0) node [bnode] {} -- (\x+0.5,0.5) node [bnode] {} --(\x+1,0)node [bnode] {}-- (\x+0.5,-0.5) node [bnode]{} --cycle;
}

\draw (2.5,0) node [color=black] {$\cdots$};
\path[fill=gray]  (4.5,0.5)  -- (4,1) node [bnode] {} --(4.5,1.5)node
[bnode] {}-- (5,1) node [bnode]{} --cycle;
\path[fill=gray]  (4.5,-0.5)  -- (4,-1) node [bnode] {} --(4.5,-1.5)node
[bnode] {}-- (5,-1) node [bnode]{} --cycle;
\draw (2.5,0) node [color=black] {$\cdots$};
\draw (6.5,0) node [color=black] {$\cdots$};
\draw (4.5, -2) node [color=black] { $H^{(4)}_{2}(n,j)$ for $j=0,1,2,3$.};
\draw (3.35,0.65) node [enode] {$\leftarrow$};
\draw (3.5,0.95) node [enode] {$n$};
\draw (5.55,0.65) node [enode] {$\rightarrow$};
\draw (5.5,0.95) node [enode] {$j$};
\draw (1,0) node [bnode][color=red] {};
 \draw (1.37,0) node [color=red] {$q_{n}$};
\draw (4.5,0.8) node [color=red] {$x_{1}$};
\draw (4.5,-0.8) node [color=red] {$x_{2}$};
\draw (3.37,0) node [color=red] {$q_{1}$};
\draw (4.5,0.3) node [color=red] {$h_{1}$};
\draw (4.6,-0.3) node [color=red] {$h_{2}$};
\draw (4.25,0) node [color=red] {$h_{4}$};
\draw (4.75,0) node [color=red] {$h_{3}$};
\draw (5.35,0) node [color=red] {$c_{j}$};
\draw (7.35,0) node [color=red] {$c_{1}$};
\end{tikzpicture}
\end{center}

{\bf Claim (a):} Both $H^{(4)}_1(n)$ and $H^{(4)}_2(n,j)$ (for
$j=0,1,2,3$) admit a $\beta$-supernormal labeling such that the label
at the corner of the red vertex is greater than $\frac{1-\sqrt{1-4\beta}}{2}$.

{\bf Proof of Claim (a):}
We will label the partial graphs so that the $\beta$-normal properties hold
except at the corner of the red vertex.
In graph $H^{4}_{1}(n)$, we set $x_{1}=x_{2}=\beta$, $y_{1}=\frac{\beta}{1-2\beta}=f_{\beta}(2\beta)$, $y_{n}=f^{n}_{\beta}(2\beta)$.
Since $2\beta\in(\frac{1-\sqrt{1-4\beta}}{2}, \frac{1+\sqrt{1-4\beta}}{2})$, by Lemma \ref{infty},
 we get that $y_{n}=f^n_\beta>\frac{1-\sqrt{1-4\beta}}{2}$.

In graph $H^{4}_{2}(n,j)$, we set $x_{1}=x_{2}=c_{1}=\beta$,
$h_{1}=h_{2}=1-\beta$, $c_{j}=f^{j-1}_{\beta}(\beta)$.
When $j=0$, we have $h_3=1$ and $h_4=\frac{\beta}{h_1h_2h_3}=\frac{\beta}{(1-\beta)^2}$.
When $j=1$, we have $h_{3}=1-\beta$ and
$h_{4}=\frac{\beta}{h_1h_2h_3}=\frac{\beta}{(1-\beta)^3}$.
When $j=2$, we have $h_{3}=1-c_2=\frac{1-2\beta}{1-\beta}$ and $h_{4}=\frac{\beta}{h_1h_2h_3}=\frac{\beta}{(1-\beta)(1-2\beta)}$.
When $j=3$, we set $h_{3}=1-c_{3}=\frac{1-3\beta+\beta^2}{1-2\beta}$ and
 $h_{4}=\frac{\beta}{h_1h_2h_3}=\frac{\beta(1-2\beta)}{(1-3\beta+\beta^2)(1-\beta)^2}$.
We can check directly that for all $j=0,1,2,3$, the value
$h_{4}\in (\frac{1-\sqrt{1-4\beta}}{2}, \frac{1+\sqrt{1-4\beta}}{2})$.
Since $q_{n}=f^n_{\beta}(h_{4})$, and thus by Lemma \ref{infty},
we get $q_{n}>\frac{1-\sqrt{1-4\beta}}{2}$.

To show $H$ must be an open quipu as stated in the theorem, we need
exclude the following structures.
\begin{enumerate}
\item We first show that all branching vertices and branching edges lie
  on the same path denoted by $P$. Otherwise, $H$ contains the following subhypergraph.
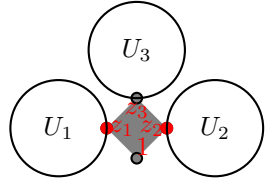
\begin{center}
\begin{tikzpicture}[thick, scale=0.8, bnode/.style={circle, draw,
    fill=black!50, inner sep=0pt, minimum width=4pt}, enode/.style={red}]
    \foreach \x in {1}
    {
    \path[fill=gray]  (\x,0) node [bnode] {} -- (\x+0.5,0.5) node [bnode] {} --(\x+1,0)node [bnode] {}-- (\x+0.5,-0.5) node [bnode]{} --cycle;
}
\draw (1,0) node [bnode][color=red] {};
\draw (2,0) node [bnode][color=red] {};
\draw (1.5,0.3) node [color=red] {$z_{3}$};
\draw (1.6,-0.3) node [color=red] {$1$};
\draw (1.25,0) node [color=red] {$z_{1}$};
\draw (1.75,0) node [color=red] {$z_{2}$};
\draw (0.2,0) ellipse (0.8);
\draw (2.8,0) ellipse (0.8);
\draw (1.5,1.3) ellipse (0.8);
\draw (0.2,0) node [color=black] {$U_{1}$};
\draw (2.8,0) node [color=black] {$U_{2}$};
\draw (1.5,1.3) node [color=black] {$U_{3}$};

\end{tikzpicture}
\end{center}
where $U_{1}$, $U_{2}$ and $U_{3}$ are chosen from $H^{4}_{1}(n)$ (for
some $n\geq 0$) and $H^{4}_{2}(n,j)$ (for some $n\geq 0$ and
$j=0,1,2,3$)
and pieces are glued through red nodes.

From Claim (a), we
 have $z_{1}\cdot z_{2}\cdot z_{3} \cdot 1<(\frac{1+\sqrt{1-4\beta}}{2})^{3}\cdot 1=\beta$.
 So, this subhypergraph is $\beta$-supernormal. It implies
 $\rho(H)>\rho_4'$.

\item  Now we show that any branch vertex  must lie at the end of that
  path $P$. Otherwise, $H$ contains the following subhypergraph.
 
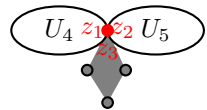
\begin{center}
\begin{tikzpicture}[thick, scale=0.8, bnode/.style={circle, draw,
    fill=black!50, inner sep=0pt, minimum width=4pt}, enode/.style={red}]
 %    \foreach \x in {1,2}
%     {
%     \path[fill=gray]  (\x,0) node [bnode] {} -- (\x+0.5,0.5) node [bnode] {} --(\x+1,0)node [bnode] {}-- (\x+0.5,-0.5) node [bnode]{} --cycle;
% }
\path[fill=gray]  (2,0)  -- (1.66,-0.66) node [bnode] {} -- (2,-1.2) node [bnode]{}--(2.3,-0.66)node [bnode] {} --cycle;
\draw (1.2,0) ellipse (0.8 and 0.4);
\draw (1.2,0) node [color=black] {$U_{4}$};
\draw (2.8,0) ellipse (0.8 and 0.4);
\draw (2.8,0) node [color=black] {$U_{5}$};
  \draw (1.75,0) node [color=red] {$z_{1}$};
  \draw (2.25,0) node [color=red] {$z_{2}$};
  \draw (2,-0.35) node [color=red] {$z_{3}$};
  \draw (2,0) node [bnode][color=red] {};
\end{tikzpicture}
\end{center}
where $U_{4}$ and $U_{5}$ are chosen from $H^{4}_{1}(n)$ (for
some $n\geq 0$) and $H^{4}_{2}(n,j)$ (for some $n\geq 0$ and
$j=0,1,2,3$) and pieces are glued through red nodes.
From Claim (a), we
 have $z_{1}+z_{2}+ z_{3}>\frac{1-\sqrt{1-4\beta}}{2} +
 \frac{1-\sqrt{1-4\beta}}{2} + \beta=1$.
 So, this subhypergraph is $\beta$-supernormal. It implies
 $\rho(H)>\rho_4'$.

\item Now we show that any branch edge must also lie at the end of that
  path $P$. Otherwise, $H$ contains the following subhypergraph.
 
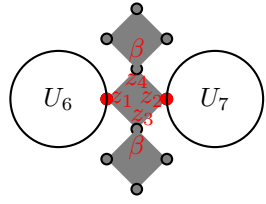
\begin{center}
\begin{tikzpicture}[thick, scale=0.8, bnode/.style={circle, draw,
    fill=black!50, inner sep=0pt, minimum width=4pt}, enode/.style={red}]
    \foreach \x in {1}
    {
    \path[fill=gray]  (\x,0) node [bnode] {} -- (\x+0.5,0.5) node [bnode] {} --(\x+1,0)node [bnode] {}-- (\x+0.5,-0.5) node [bnode]{} --cycle;
}
\path[fill=gray]  (1.5,0.5)  -- (1,1) node [bnode] {} --(1.5,1.5)node
[bnode] {}-- (2,1) node [bnode]{} --cycle;
\path[fill=gray]  (1.5,-0.5)  -- (1,-1) node [bnode] {} --(1.5,-1.5)node
[bnode] {}-- (2,-1) node [bnode]{} --cycle;
\draw (1,0) node [bnode][color=red] {};
 \draw (2,0) node [bnode][color=red] {};
\draw (1.5,0.8) node [color=red] {$\beta$};
\draw (1.5,-0.8) node [color=red] {$\beta$};
\draw (1.5,0.3) node [color=red] {$z_{4}$};
\draw (1.6,-0.3) node [color=red] {$z_{3}$};
\draw (1.25,0) node [color=red] {$z_{1}$};
\draw (1.75,0) node [color=red] {$z_{2}$};
\draw (0.2,0) ellipse (0.8);
\draw (2.8,0) ellipse (0.8);
\draw (0.2,0) node [color=black] {$U_{6}$};
\draw (2.8,0) node [color=black] {$U_{7}$};
\end{tikzpicture}
\end{center}
where $U_{6}$ and $U_{7}$ are chosen from $H^{4}_{1}(n)$ (for
some $n\geq 0$) and $H^{4}_{2}(n,j)$ (for some $n\geq 0$ and
$j=0,1,2,3$)
and pieces are glued through red nodes.
We have
$$z_{1}\cdot z_{2}\cdot z_{3}\cdot z_{4}<(\frac{1+\sqrt{1-4\beta}}{2})^{2}\cdot (1-\beta)^{2}\approx 0.2229<\beta.$$
This subhypergraph is $\beta$-supernormal. Thus we have
$\rho(H)>\rho'_r$. Contradiction.

\item It remains to show that each 4-branching edge is attached by
  three paths of length $1$, $1$, and $k$ ($k = 1, 2, 3$) respectively if it is
  not a $4$-dagger. Otherwise, it contains one of the following two
  hypergraphs as a subhypergraph.

 
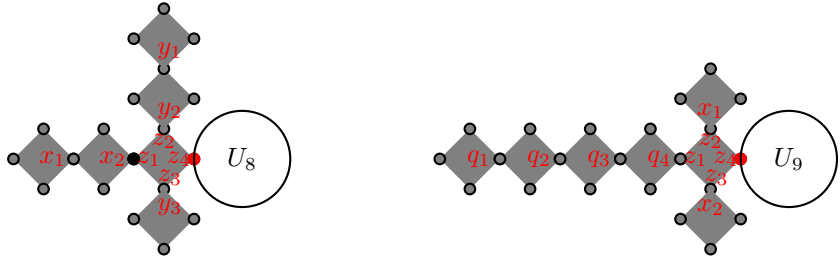
\begin{center}
\begin{tikzpicture}[thick, scale=0.8, bnode/.style={circle, draw,
    fill=black!50, inner sep=0pt, minimum width=4pt}, enode/.style={red}]
    \foreach \x in {-1,0,1}
    {
    \path[fill=gray]  (\x,0) node [bnode] {} -- (\x+0.5,0.5) node [bnode] {} --(\x+1,0)node [bnode] {}-- (\x+0.5,-0.5) node [bnode]{} --cycle;
}
\path[fill=gray]  (1.5,0.5)  -- (1,1) node [bnode] {} --(1.5,1.5)node
[bnode] {}-- (2,1) node [bnode]{} --cycle;
\path[fill=gray]  (1.5,-0.5)  -- (1,-1) node [bnode] {} --(1.5,-1.5)node
[bnode] {}-- (2,-1) node [bnode]{} --cycle;
\path[fill=gray]  (1.5,1.5)  -- (1,2) node [bnode] {} --(1.5,2.5)node
[bnode] {}-- (2,2) node [bnode]{} --cycle;
\draw (1,0) node [bnode][color=black] {};
 \draw (2,0) node [bnode][color=red] {};
\draw (1.5,0.3) node [color=red] {$z_{2}$};
\draw (1.6,-0.3) node [color=red] {$z_{3}$};
\draw (1.25,0) node [color=red] {$z_{1}$};
\draw (1.75,0) node [color=red] {$z_{4}$};
\draw (2.8,0) ellipse (0.8);
\draw (2.8,0) node [color=black] {$U_{8}$};
\draw (-0.35,0) node [color=red] {$x_{1}$};
\draw (0.65,0) node [color=red] {$x_{2}$};
\draw (1.6,0.8) node [color=red] {$y_{2}$};
\draw (1.6,1.8) node [color=red] {$y_{1}$};
\draw (1.6,-0.8) node [color=red] {$y_{3}$};
\end{tikzpicture} \hfil
\begin{tikzpicture}[thick, scale=0.8, bnode/.style={circle, draw,
    fill=black!50, inner sep=0pt, minimum width=4pt}, enode/.style={red}]
    \foreach \x in {0,1,2,3,4}
    {
    \path[fill=gray]  (\x,0) node [bnode] {} -- (\x+0.5,0.5) node [bnode] {} --(\x+1,0)node [bnode] {}-- (\x+0.5,-0.5) node [bnode]{} --cycle;
}
\path[fill=gray]  (4.5,0.5)  -- (4,1) node [bnode] {} --(4.5,1.5)node
[bnode] {}-- (5,1) node [bnode]{} --cycle;
\path[fill=gray]  (4.5,-0.5)  -- (4,-1) node [bnode] {} --(4.5,-1.5)node
[bnode] {}-- (5,-1) node [bnode]{} --cycle;
\draw (5,0) node [bnode][color=red] {};
 \draw (0.65,0) node [color=red] {$q_{1}$};
  \draw (1.65,0) node [color=red] {$q_{2}$};
  \draw (2.65,0) node [color=red] {$q_{3}$};
  \draw (3.65,0) node [color=red] {$q_{4}$};
\draw (4.5,0.8) node [color=red] {$x_{1}$};
\draw (4.5,-0.8) node [color=red] {$x_{2}$};
\draw (4.5,0.3) node [color=red] {$z_{2}$};
\draw (4.6,-0.3) node [color=red] {$z_{3}$};
\draw (4.25,0) node [color=red] {$z_{1}$};
\draw (4.75,0) node [color=red] {$z_{4}$};
\draw (5.8,0) ellipse (0.8);
\draw (5.8,0) node [color=black] {$U_{9}$};
\end{tikzpicture}
\end{center}
where $U_{8}$ and $U_{9}$ are chosen from $H^{4}_{1}(n)$ (for
some $n\geq 0$) and $H^{4}_{2}(n,j)$ (for some $n\geq 0$ and
$j=0,1,2,3$)
and pieces are glued through red nodes.

For the left hypergraph, we set $x_{1}=y_{1}=y_{3}=\beta$,
$x_{2}=y_{2}=\frac{\beta}{1-\beta}$, $z_3=1-\beta$, and
$z_{1}=z_{2}=\frac{1-2\beta}{1-\beta}$, and
$z_4<\frac{1+\sqrt{1-4\beta}}{2}$ (from Claim (a)).
Thus, the product of labels on the branching edge is
$$z_{1}\cdot z_{2}\cdot z_{3}\cdot z_{4}<(\frac{1+\sqrt{1-4\beta}}{2})\cdot (\frac{1-2\beta}{1-\beta})^{2}\cdot (1-\beta)\approx  0.2254<\beta.$$
For the right hypergraph,  we set
$q_{1}=x_{1}=x_{2}=\beta$, $q_{i}=f_{\beta}^{i-1}(\beta)$ ($i=2,3,4$),
$z_{1}=1-q_{4}=\frac{1-4\beta+3\beta^2}{1-3\beta+\beta^2}$, $z_{2}=z_{3}=1-\beta$, and
$z_4<\frac{1+\sqrt{1-4\beta}}{2}$ (from Claim (a)).

Thus, the product of labels on the branching edge is
$$z_{1}\cdot z_{2}\cdot z_{3}\cdot
z_{4}<\frac{1+\sqrt{1-4\beta}}{2}\cdot \frac{1-4\beta+3\beta^2}{1-3\beta+\beta^2}
 \cdot (1-\beta)^{2}\approx 0.2314<\beta.$$
Thus the both hypergraphs above are $\beta$-supernormal. Thus we have
$\rho(H)>\rho'_r$. Contradiction.
Therefore $H$ must be an open quipu as stated in the theorem.
\end{enumerate}

\noindent {\bf Case 3.} $H$ is the $4$-dagger $H^{(4)}_{i,j,k,l}$ for $1\leq
i\leq j\leq k\leq l.$

We try to label $H^{(4)}_{i,j,k,l}$ so that the $\beta$-normal
properties hold except the product of the labels at the branching
edge. Not that the  product of the labels at the branching
edge, denoted by $g(i,j,k,l)$,  is given by
$$g(i,j,k,l)=f^{i-1}_\beta(\beta)f^{j-1}_\beta(\beta)f^{k-1}_\beta(\beta)f^{l-1}_\beta(\beta).$$
It is easy to verify that
$g(i,j,k,l)<\beta$ for $(i,j,k,l)=$ $(2,2,2,2)$, $(1,2,2,4)$,
$(1,2,3,3)$, $(1,1,5,5)$, $(1,1,4,6)$. $H$ cannot contain those
$4$-daggers as a subhypergraph.
Therefore,  $H$ must be one
  of the following hypergraphs $H^{(4)}_{1,2,2,2}$,
  $H^{(4)}_{1,2,2,3}$, $H^{(4)}_{1,1,4,4}$, $H^{(4)}_{1,1,4,5}$, and
$H^{(4)}_{1,1,k,l}$ ($1\leq k\leq 3$, and $k\leq l$). It is also easy
to verify that those $4$-daggers are $\beta$-subnormal. So this is a
complete list of $4$-daggers with $\rho(H)<\rho_4'$.

\end{proof}

\section{Proof of Theorem \ref{t3}}
\begin{proof} Let the {\em edge-star} $S_r^{(r)}$ be the $r$-uniform
  hypergraph consisting of
  $r+1$ edges: $F_0=\{v_1,v_2,\ldots,v_r\}$, $F_1$, $\ldots$, $F_r$,
  where each $F_i\cap F_0=\{v_i\}$ for
  $1\leq i\leq r$, and $F_i\cap
  F_j=\emptyset$ for $1\leq i\leq j\leq r$. (See the picture of
  $S_5^{(5)}$ at Theorem \ref{t3}.)

We first show that $\rho_r(S_r^{(r)})>\rho_r'$ for $r\geq 6$.
This can be done by assigning $B(v_i,F_i)=\beta$ and $B(v_i,
F_0)=1-\beta$, for $1\leq i\leq r$.
Note that the product of labels on $F_0$ is
$$(1-\beta)^r<\beta$$
for all $r\geq 6$. Thus, $S_r^{(r)}$ is $\beta$-supernormal.
If there is an irreducible $r$-uniform hypergraph $H$ with $\rho(H)\leq
\rho_r'$ for $r\geq 6$, then $H$ contains a sub-homomorphic type $S_r^{(r)}$.
By Lemma \ref{subhomo}, we have
$\rho(H)\geq \rho(S_r^{(r)})>\rho_r'$, contradiction.

The same argument shows that $S_5^{(5)}$ is $\beta$-subnormal.
Let $H$ be an irreducible $5$-uniform hypergraph $H$ with $\rho(H)\leq
\rho_5'$. If $H$ is not $S_5^{(5)}$, $H$ contains one of the following
sub-homomorphic types ${S'}_5^{(5)}$ and $S_{5+}^{(5)}$.

\begin{center}
\begin{tikzpicture}[thick, scale=0.8, bnode/.style={circle, draw,
    fill=black!50, inner sep=0pt, minimum width=4pt}, enode/.style={red}]
  \path[fill=gray]  (0,0) node [bnode] {} -- (0.5,0.75) node [bnode] {} --(1,0)node [bnode] {}-- (1,-0.75) node [bnode]{}-- (0,-0.75) node [bnode]{}--cycle;
  \path[fill=gray]  (0.5,0.75) node [bnode] {} -- (0,1.5) node [bnode] {} --(0,2.25)node [bnode] {}-- (1,2.25) node [bnode]{}-- (1,1.5) node [bnode]{}--cycle;
 \path[fill=gray]  (0,0) node [bnode] {} -- (-0.5,0.5) node [bnode] {} --(-1.25,0.5)node [bnode] {}-- (-1.25,-0.5) node [bnode]{}-- (-0.5,-0.5) node [bnode]{}--cycle;
 \path[fill=gray]  (0,-0.75) node [bnode] {} -- (-0.5,-1.5) node [bnode] {} --(-0.5,-2.25)node [bnode] {}-- (0.5,-2.25) node [bnode]{}-- (0.5,-1.5) node [bnode]{}--cycle;
  \path[fill=gray]  (1,0) node [bnode] {} -- (1.5,0.5) node [bnode] {} --(2.25,0.5)node [bnode] {}-- (2.25,-0.5) node [bnode]{}-- (1.5,-0.5) node [bnode]{}--cycle;
 \path[fill=gray]  (1,-0.75) node [bnode] {} -- (0.8,-1.35) node
 [bnode] {} --(1.5,-2) node [bnode] {}-- (2,-1.3) node [bnode]{}-- (1.5,-0.5) --cycle;
\draw (0.3,-0.2) node [color=red] {$x_{4}$};
\draw (0.4,0.5) node [color=red] {$x_{5}$};
\draw (0.7,0) node [color=red] {$x_{1}$};
\draw (0.8,-0.5) node [color=red] {$x_{2}$};
\draw (0.3,-0.6) node [color=red] {$x_{3}$};

\draw(1.6,-0.25) node [color=red] {$\frac{1}{2}$};
\draw(1.6,-0.9) node [color=red] {$\frac{1}{2}$};

\draw (0.5, -3) node {${S'}_5^{(5)}$};
\end{tikzpicture}
\hfil
\begin{tikzpicture}[thick, scale=0.8, bnode/.style={circle, draw,
    fill=black!50, inner sep=0pt, minimum width=4pt}, enode/.style={red}]
  \path[fill=gray]  (0,0) node [bnode] {} -- (0.5,0.75) node [bnode] {} --(1,0)node [bnode] {}-- (1,-0.75) node [bnode]{}-- (0,-0.75) node [bnode]{}--cycle;
  \path[fill=gray]  (0.5,0.75) node [bnode] {} -- (0,1.5) node [bnode] {} --(0,2.25)node [bnode] {}-- (1,2.25) node [bnode]{}-- (1,1.5) node [bnode]{}--cycle;
 \path[fill=gray]  (0,0) node [bnode] {} -- (-0.5,0.5) node [bnode] {} --(-1.25,0.5)node [bnode] {}-- (-1.25,-0.5) node [bnode]{}-- (-0.5,-0.5) node [bnode]{}--cycle;
 \path[fill=gray]  (0,-0.75) node [bnode] {} -- (-0.68,-1.5) node [bnode] {} --(-0.88,-2.25)node [bnode] {}-- (0.12,-2.45) node [bnode]{}-- (0.2,-1.65) node [bnode]{}--cycle;
  \path[fill=gray]  (1,0) node [bnode] {} -- (1.5,0.5) node [bnode] {} --(2.25,0.5)node [bnode] {}-- (2.25,-0.5) node [bnode]{}-- (1.5,-0.5) node [bnode]{}--cycle;
 \path[fill=gray]  (1,-0.75) node [bnode] {} -- (0.8,-1.65) node
 [bnode] {} --(0.88,-2.45)node [bnode] {}-- (1.88,-2.25) node
 [bnode]{}-- (1.68,-1.5) node [bnode]{}--cycle;

 \path[fill=gray]  (2.25,-0.5) node [bnode] {} -- (2.75,0) node [bnode] {} --(3.5,0)node [bnode] {}-- (3.5,-1) node [bnode]{}-- (2.75,-1) node [bnode]{}--cycle;

\draw (0.3,-0.2) node [color=red] {$y_{4}$};
\draw (0.4,0.5) node [color=red] {$y_{5}$};
\draw (0.7,0) node [color=red] {$y_{1}$};
\draw (0.8,-0.5) node [color=red] {$y_{2}$};
\draw (0.3,-0.6) node [color=red] {$y_{3}$};
\draw (2.55,-0.5) node [color=red] {$\beta$};

\draw (0.5, -3) node {${S}_{5+}^{(5)}$};
\end{tikzpicture}
\end{center}
For ${S'}_5^{(5)}$, we can label the corner of the only identified
vertex not on the branching edge by $\frac{1}{2}$,  and set
$x_1=x_2=1-2\beta$, $x_3=x_4=x_5=1-\beta$. We can check that the
product of labels on the branching edge is
$$x_1x_2x_3x_4x_5=(1-2\beta)^2(1-\beta)^3\approx 0.1242<\beta.$$
For ${S}_{5+}^{(5)}$, we can set
$y_1=1-f_{\beta}(\beta)=\frac{1-2\beta}{1-\beta}$, $y_2=y_3=y_4=y_5=1-\beta$.
We can check that the
product of labels on the branching edge is
$$y_1y_2y_3y_4y_5=(1-2\beta)(1-\beta)^4\approx 0.1798<\beta.$$
Thus, both ${S'}_5^{(5)}$ and ${S}_{5+}^{(5)}$ are consistently
$\beta$-supernormal. This implies that $\rho(H)>\rho_5'$, contradiction.
Thus $H$ must be the five edge-star.
\end{proof}

\section{Constructing open quipus and closed quipus with $\rho(H)\leq (r-1)!\sqrt[r]{2+\sqrt{5}}$}
In this paper, we give a description of the connected $r$-uniform hypergraphs with
spectral radius at most $(r-1)!\sqrt[r]{2+\sqrt{5}}$: they are
extended from the irreducible ones listed in Theorems
\ref{t1}-\ref{t3} and the $2$-graphs listed by Cvetkovi\'c et al
\cite{CDG} and Brouwer-Neumaier \cite{BN}. This is not a complete
description for $r\geq 3$, but rather a coarse description.
The scenario is similar to the results of Woo and Neumaier on
the graphs with spectral radius at most $\frac{3}{2}\sqrt{2}$ (see
\cite{WN}). Our method is very different from the linear algebra method
used by Woo and Neumaier. In fact, it is possible to simply the proof
of Woo-Neumaier's result using our new method but we will omit it here.

In the rest of this section, we will construct  many examples with
$\rho(H)\leq (r-1)!\sqrt[r]{2+\sqrt{5}}$. This shows that the
descriptions in Theorem \ref{t1}-\ref{t3} are somewhat tight.

The $4$-daggers are completely classified so no construction is
needed. We only need to construct closed $3$- quipus, open $3$-quipus and open $4$-quipus first.
The idea is to present some partial hypergraphs, which can be glued
together to form a hypergraph with $\rho(H)\leq
(r-1)!\sqrt[r]{2+\sqrt{5}}$. A partial $r$-uniform hypergraph
is an $r$-uniform hypergraph together with (one or two) designated
vertex/vertices. A partial hypergraph $H$ is called $\alpha$-subnormal
if there exists a weighted
incidence matrix $B$ satisfying
\begin{enumerate}
\item $\prod_{v\in e}B(v,e)\geq \alpha$,  for any $e\in E(H)$.
\item $\sum_{e\colon v\in e}B(v,e)\leq \frac{1}{2}$, for any
  designated vertex $v$,
\item $\sum_{e\colon v\in e}B(v,e)\leq 1$, for any non-designated vertex.
\end{enumerate}

\begin{lemma} Consider the following partial hypergraphs  $G_1^{(3)}(m,k_1,k_2)$, $G_2^{(2)}(m,k)$,
and $G_3^{(4)}(t,k)$ (with designated   vertices colored in red). We have
\begin{enumerate}
\item  For any $m\geq 1$, there exists a $k_0$ such that for
  any $k_1,k_2\geq k_0$, $G_1^{(3)}(m,k_1,k_2)$ is $(\sqrt{5}-2)$-subnormal.
\item  For any $m\geq 1$, there exists a $k_0$ such that for
  any $k\geq k_0$, $G_2^{(2)}(m,k)$ is $(\sqrt{5}-2)$-subnormal.
\item  For any $t=1,2,3$, there exists a $k_t$ such that for
  any $k\geq k_t$, $G_3^{(4)}(t,k)$ is $(\sqrt{5}-2)$-subnormal.
\end{enumerate}
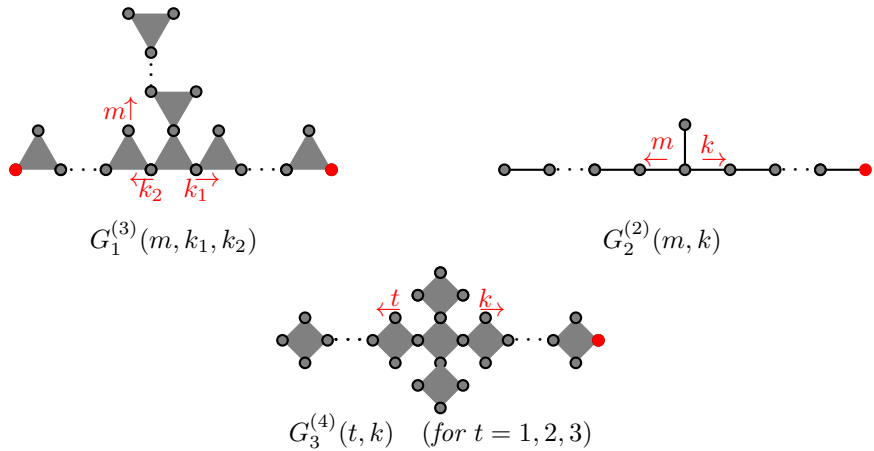
\begin{center}
\begin{tikzpicture}[thick, scale=0.6, bnode/.style={circle, draw,
    fill=black!50, inner sep=0pt, minimum width=4pt}, enode/.style={color=red}]
\foreach \x in {4}
    {
\path[fill=gray]  (\x+0.5,0.866) node [bnode] {}-- (\x,1.732) node [bnode] {} --
    (\x+1,1.732)  node [bnode] {} --cycle;
  \draw (\x,2.28) node [color=black] {$\vdots$};

\path[fill=gray]  (\x,2.598) node [bnode] {}-- (\x-0.5,3.464) node [bnode] {} --
    (\x+0.5,3.464)  node [bnode] {} --cycle;

}
\foreach \x in {1,3,4,5,7}
    {
    \path[fill=gray]  (\x,0) node [bnode] {} -- (\x+0.5,0.866) node [bnode] {} --(\x+1,0)node [bnode] {}--cycle;
}
\draw (6.5,0) node [] {$\cdots$};

\draw (2,0) node  [bnode] {};
\draw (4,0) node  [bnode] {};
\draw (7,0) node  [bnode] {};
\draw (2.5,0) node []{$\ldots$};
\draw (3.20,1.25) node [enode] {$m$};
\draw (3.55,1.35) node [enode] {$\uparrow$};
\draw (3.8,-0.25) node [enode] {$\leftarrow$};
\draw (5.25,-0.25) node [enode] {$\rightarrow$};
\draw (4.0,-0.45) node [enode] {$k_{2}$};
\draw (5.0,-0.45) node [enode] {$k_{1}$};
\draw (1,0) node  [bnode][color=red] {};
\draw (8,0) node  [bnode][color=red] {};
\draw (4.5, -1.5) node [color=black] { $G_{1}^{(3)}(m,k_1,k_2)$};
\end{tikzpicture}
\hfil
\begin{tikzpicture}[thick, scale=0.6, bnode/.style={circle, draw,
    fill=black!50, inner sep=0pt, minimum width=4pt}, enode/.style={red}]

\draw (5,0) -- (5, 1) node [bnode] {};

\foreach \x in {1,3,4,5,6,8}
    {
    \draw  (\x,0) node [bnode] {}  --(\x+1,0);
}

\draw (2.5,0) node [color=black] {$\cdots$};
\draw (7.5,0) node [color=black] {$\cdots$};
\draw (3,0) node  [bnode] {};
\draw (8,0) node  [bnode] {};
\draw (7,0) node  [bnode] {};
\draw (2,0) node  [bnode] {};
\draw (9,0) node  [bnode][color=red] {};
\draw (4.35,0.2) node [enode] {$\leftarrow$};
\draw (5.65,0.2) node [enode] {$\rightarrow$};
\draw (4.5,0.58) node [enode] {$m$};
\draw (5.5,0.58) node [enode] {$k$};
\draw (4.5, -1.5) node [color=black] { $G_{2}^{(2)}(m,k)$};
\end{tikzpicture}

\begin{tikzpicture}[thick, scale=0.6, bnode/.style={circle, draw,
    fill=black!50, inner sep=0pt, minimum width=4pt}, enode/.style={red}]
    \foreach \x in {1,3,4,5,7}
    {
    \path[fill=gray]  (\x,0) node [bnode] {} -- (\x+0.5,0.5) node [bnode] {} --(\x+1,0)node [bnode] {}-- (\x+0.5,-0.5) node [bnode]{} --cycle;
}

\draw (2.5,0) node [color=black] {$\cdots$};
\path[fill=gray]  (4.5,0.5)  -- (4,1) node [bnode] {} --(4.5,1.5)node
[bnode] {}-- (5,1) node [bnode]{} --cycle;
\path[fill=gray]  (4.5,-0.5)  -- (4,-1) node [bnode] {} --(4.5,-1.5)node
[bnode] {}-- (5,-1) node [bnode]{} --cycle;
\draw (2.5,0) node [color=black] {$\cdots$};
\draw (6.5,0) node [color=black] {$\cdots$};
\draw (4.5, -2) node [color=black] { $G_{3}^{(4)}(t,k)\quad (\mbox{for } t=1,2,3)$};
\draw (3.35,0.65) node [enode] {$\leftarrow$};
\draw (3.5,0.95) node [enode] {$t$};
\draw (5.65,0.65) node [enode] {$\rightarrow$};
\draw (5.5,0.95) node [enode] {$k$};
\draw (8,0) node [bnode][color=red] {};
\foreach \x in {4.5}
{
\path[fill=gray]  (\x,0.5)node [bnode] {}  -- (\x-0.5,1) node [bnode] {} --(\x,1.5)node
[bnode] {}-- (\x+0.5,1) node [bnode]{} --cycle;
}
\end{tikzpicture}
\end{center}
 \end{lemma}
\begin{proof}
We label the corner of the designated vertices by $\frac{1}{2}$ and
the corner of other leaf-vertices by $1$. We try to maintain the
properties that the product of all labels in one edge is $\beta$ and
the sum of all labels at one vertex is $1$ except at the branching
vertex or at the branching edge. We get the labels of
the three partial graphs as follows
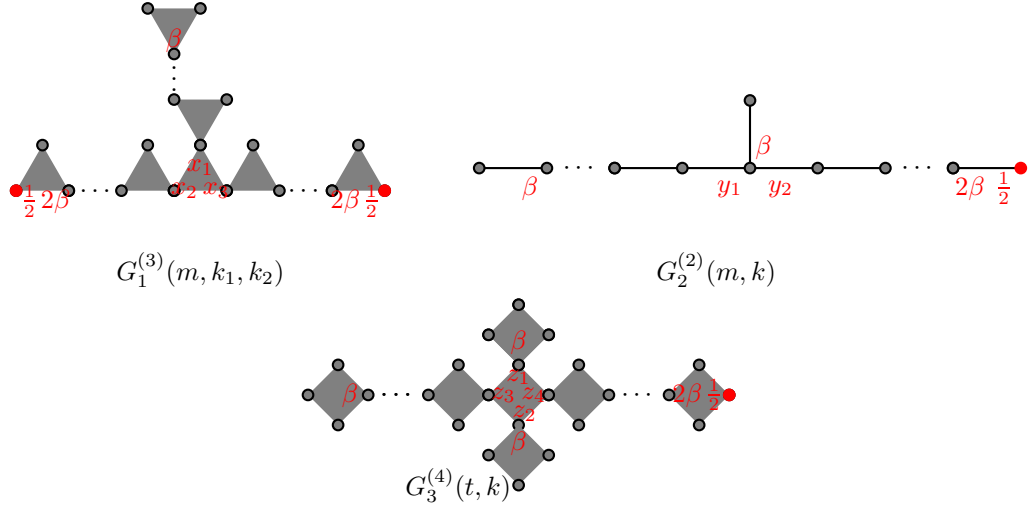
\begin{center}
\begin{tikzpicture}[thick, scale=0.7, bnode/.style={circle, draw,
    fill=black!50, inner sep=0pt, minimum width=4pt}, enode/.style={color=red}]
\foreach \x in {4}
    {
\path[fill=gray]  (\x+0.5,0.866) node [bnode] {}-- (\x,1.732) node [bnode] {} --
    (\x+1,1.732)  node [bnode] {} --cycle;
  \draw (\x,2.28) node [color=black] {$\vdots$};

\path[fill=gray]  (\x,2.598) node [bnode] {}-- (\x-0.5,3.464) node [bnode] {} --
    (\x+0.5,3.464)  node [bnode] {} --cycle;

}
\foreach \x in {1,3,4,5,7}
    {
    \path[fill=gray]  (\x,0) node [bnode] {} -- (\x+0.5,0.866) node [bnode] {} --(\x+1,0)node [bnode] {}--cycle;
}
\draw (6.5,0) node [] {$\cdots$};

\draw (2,0) node  [bnode] {};
\draw (4,0) node  [bnode] {};
\draw (7,0) node  [bnode] {};
\draw (2.5,0) node []{$\ldots$};
\draw (1,0) node  [bnode][color=red] {};
\draw (8,0) node  [bnode][color=red] {};
\draw (4.5, -1.5) node [color=black] { $G_{1}^{(3)}(m,k_1,k_2)$};
\draw (1.75,-0.2) node [color=red] {$2\beta$};
%\draw (3.75,-0.2) node [color=red] {$x_{k_{2}}$};
\draw (4.8,0) node [color=red] {$x_{3}$};
\draw (4.2,0) node [color=red] {$x_{2}$};
\draw (4.5,0.45) node [color=red] {$x_{1}$};
%\draw (4.5,1.2) node [color=red] {$y_{m}$};
\draw (4.0,2.85) node [color=red] {$\beta$};
%\draw (5.45,-0.2) node [color=red] {$u_{k_2}$};
\draw (7.25,-0.2) node [color=red] {$2\beta$};
\draw (7.75,-0.2) node [color=red] {$\frac{1}{2}$};
\draw (1.25,-0.2) node [color=red] {$\frac{1}{2}$};
\end{tikzpicture}
\hfil
\begin{tikzpicture}[thick, scale=0.9, bnode/.style={circle, draw,
    fill=black!50, inner sep=0pt, minimum width=4pt}, enode/.style={red}]

\draw (5,0) -- (5, 1) node [bnode] {};

\foreach \x in {1,3,4,5,6,8}
    {
    \draw  (\x,0) node [bnode] {}  --(\x+1,0);
}

\draw (2.5,0) node [color=black] {$\cdots$};
\draw (7.5,0) node [color=black] {$\cdots$};
\draw (3,0) node  [bnode] {};
\draw (8,0) node  [bnode] {};
\draw (7,0) node  [bnode] {};
\draw (2,0) node  [bnode] {};
\draw (9,0) node  [bnode][color=red] {};
%\draw (4.35,0.2) node [enode] {$\leftarrow$};
%\draw (5.65,0.2) node [enode] {$\rightarrow$};
%\draw (4.5,0.58) node [enode] {$m$};
%\draw (5.5,0.58) node [enode] {$k$};
\draw (4.5, -1.5) node [color=black] { $G_{2}^{(2)}(m,k)$};
\draw (1.75,-0.3) node [color=red] {$\beta$};
\draw (4.7,-0.3) node [color=red] {$y_1$};
\draw (5.45,-0.3) node [color=red] {$y_2$};
\draw (8.25,-0.3) node [color=red] {$2\beta$};
\draw (8.75,-0.3) node [color=red] {$\frac{1}{2}$};
\draw (5.2, 0.3) node [enode] {$\beta$};
\end{tikzpicture}

\begin{tikzpicture}[thick, scale=0.8, bnode/.style={circle, draw,
    fill=black!50, inner sep=0pt, minimum width=4pt}, enode/.style={red}]
    \foreach \x in {1,3,4,5,7}
    {
    \path[fill=gray]  (\x,0) node [bnode] {} -- (\x+0.5,0.5) node [bnode] {} --(\x+1,0)node [bnode] {}-- (\x+0.5,-0.5) node [bnode]{} --cycle;
}

\draw (2.5,0) node [color=black] {$\cdots$};
\path[fill=gray]  (4.5,0.5)  -- (4,1) node [bnode] {} --(4.5,1.5)node
[bnode] {}-- (5,1) node [bnode]{} --cycle;
\path[fill=gray]  (4.5,-0.5)  -- (4,-1) node [bnode] {} --(4.5,-1.5)node
[bnode] {}-- (5,-1) node [bnode]{} --cycle;
\draw (2.5,0) node [color=black] {$\cdots$};
\draw (6.5,0) node [color=black] {$\cdots$};
\draw (3.5, -1.5) node [color=black] { $G_{3}^{(4)}(t,k)$};
\draw (8,0) node [bnode][color=red] {};
\foreach \x in {4.5}
{
\path[fill=gray]  (\x,0.5)node [bnode] {}  -- (\x-0.5,1) node [bnode] {} --(\x,1.5)node
[bnode] {}-- (\x+0.5,1) node [bnode]{} --cycle;
}
 \draw (1.7,0) node [color=red] {$\beta$};
\draw (4.5,0.88) node [color=red] {$\beta$};
\draw (4.5,-0.8) node [color=red] {$\beta$};
%\draw (3.65,0) node [color=red] {$w_{t}$};
\draw (4.5,0.3) node [color=red] {$z_{1}$};
\draw (4.6,-0.3) node [color=red] {$z_{2}$};
\draw (4.25,0) node [color=red] {$z_{3}$};
\draw (4.75,0) node [color=red] {$z_{4}$};
%\draw (5.45,0) node [color=red] {$u_{k_{1}}$};
\draw (7.3,0) node [color=red] {$2\beta$};
\draw (7.75,0) node [color=red] {$\frac{1}{2}$};
\end{tikzpicture}
\end{center}
Now we consider the first partial hypergraph $G_1^{(3)}$. Using the function $f_\beta$, we have  $x_1=1-f^{m-1}_\beta(\beta)$,
$x_{2}=1-f^{k_{1}-1}_{\beta}(2\beta)$,  and
$x_3=1-f^{k_{2}-1}_{\beta}(2\beta)$.
The product of the labels on the central branching edges, denoted by $g(m,k_1,k_2)$,
 satisfies
$$
g(m,k_1,k_2) =  x_1x_2x_3= (1-f^{m-1}_{\beta}(\beta)) (1-f^{k_{1}-1}_{\beta}(2\beta))(1-f^{k_2-1}_{\beta}(2\beta)).
$$
By Lemma \ref{infty},
$1-f^{m-1}_{\beta}(\beta)>\frac{1+\sqrt{1-4\beta}}{2}$,
and $\lim_{k_1\to\infty}(1-f^{k_1-1}_{\beta}(2\beta))=\lim_{k_2\to\infty}(1-f^{k_2-1}_{\beta}(2\beta))
=\frac{1+\sqrt{1-4\beta}}{2}$ since $2\beta\in
(\frac{1-\sqrt{1-4\beta}}{2}, \frac{1+\sqrt{1-4\beta}}{2})$.
Thus, $$\lim_{k_1,k_2\to\infty} g(m,k_1,k_2)
>\left(\frac{1+\sqrt{1-4\beta}}{2}\right)^3=\beta.$$
There exists a $k_0$ such that for $k_1,k_{2}\geq k_0$,
$g(m,k_1,k_2)>\beta$. I.e., $G_1^{(3)}$ is $\beta$-subnormal.

Similar argument works for the graph $G_{2}^{(2)}$. We have
$y_1=f^{m-1}_\beta(\beta)$ and $y_2=f^{k-1}_\beta(2\beta)$.
 The sum of the labels at the branching vertex is
$$\beta+y_1+y_{2}=\beta+f^{m-1}_{\beta}(\beta)+f^{k-1}_{\beta}(2\beta).$$
Note that the limit of this sum as $k$ goes to the infinity satisfies
$$\lim_{k\rightarrow \infty}(\beta+f^{m-1}_{\beta}(\beta) +f^{k-1}_{\beta}(2\beta))<
\beta+ \frac{1-\sqrt{1-4\beta}}{2} +\frac{1-\sqrt{1-4\beta}}{2}=1.$$
Thus, there exists a $k_0=k_0(m)$ such that for
any $k\geq k_0$, we get $y_{1}+y_{2}+\beta<1$. So $G_{2}^{(2)}$ is $\beta$-subnormal.

In graph $G_{3}^{(4)}(t,k)$, we have $z_1=z_2=1-\beta$,
$z_3=1-f^{t-1}_\beta(\beta)$, $z_4=1-f^{k-1}_\beta(2\beta)$.
 The product of the labels at the branching edge is
$$z_1z_2z_3z_4=(1-\beta)^2(1-f^{t-1}_\beta(\beta))(1-f^{k-1}_\beta(2\beta)).$$
For each $t=1,2,3$, it is easy to check
$$(1-\beta)^2(1-f^{t-1}_\beta(\beta))
\frac{1+\sqrt{1-4\beta}}{2}<\beta.$$
There exists a $k_t$ such that for
any $k\geq k_t$,  $G_{3}^{(4)}$ is $\beta$-subnormal.
\end{proof}

The extension also works for partial hypergraphs: add one vertex
to each edge while keep the designated vertices being
designated. Observe that  if a partial hypergraph $H$ is
$\alpha$-subnormal then so is the extension of $H$.
For any $r\geq 4$, we can extend $G_{1}^{(3)}(m,k_1,k_2)$ to
$G_{1}^{(r)}(m,k_1,k_2)$, $G_{2}^{(2)}(m,k)$ to
$G_{2}^{(r)}(m,k)$,
and $G_{3}^{(4)}(t,k)$ to $G_{3}^{(r)}(t,k)$,  glue $G_{1}^{(r)}$,
$G_{2}^{(r)}$ and $G_{3}^{(r)}$ together via the designated vertices,
and get a new graph $H$ that is still $(\sqrt{5}-2)$-subnormal. We can get
many examples of $H$ with $\rho(H)<\rho'_r$.
\begin{center}
\begin{tikzpicture}[thick, scale=0.45, bnode/.style={circle, draw,
    fill=black!50, inner sep=0pt, minimum width=4pt}, enode/.style={color=red}]
\foreach \x in {-30,30,90,150,210,270}
    {
    \path[fill=gray]  (\x-15:2) node [bnode] {} -- (\x:3) node [bnode] {} --(\x+15:2) node [bnode] {}--cycle;
}
\draw (255:2) node [bnode] {};
\draw (60:2) node [color=black] {$\cdot$};
\draw (55:2) node [color=black] {$\cdot$};
\draw (65:2) node [color=black] {$\cdot$};
\draw (120:2) node [color=black] {$\cdot$};
\draw (125:2) node [color=black] {$\cdot$};
\draw (115:2) node [color=black] {$\cdot$};
\draw (98:4.6) node [color=black] {$\vdots$};
\draw (180:2) node [color=black] {$\vdots$};
\draw (0:2) node [color=black] {$\vdots$};
\path[fill=gray]  (98:5) node [bnode] {} -- (102:6) node [bnode] {} --(93:5.877)node [bnode] {}--cycle;
\path[fill=gray]  (90:3) node [bnode] {} -- (98:4) node [bnode] {} --(82:4)node [bnode] {}--cycle;
\path[fill=gray]  (210:3) node [bnode] {} -- (202:4) node [bnode] {} --(215:4)node [bnode] {}--cycle;
\draw (200:4.6) node [color=black] {$\ldots$};
\path[fill=gray]  (198:5.2) node [bnode] {} -- (193:6.0) node [bnode] {} --(205:6.0)node [bnode] {}--cycle;
\draw (-10:4.5) node [enode] {$m$};
\draw (77:4.5) node [enode] {$m$};
\draw (193:4.5) node [enode] {$m$};
\draw (45:2) node [bnode][color=red] {};
\draw (135:2) node [bnode][color=red] {};
\draw (240:2) node [color=black] {$\cdot$};
\draw (235:2) node [color=black] {$\cdot$};
\draw (245:2) node [color=black] {$\cdot$};
\draw (300:2) node [color=black] {$\cdot$};
\draw (305:2) node [color=black] {$\cdot$};
\draw (295:2) node [color=black] {$\cdot$};
\draw (285:2) node [bnode][color=red] {};

\path[fill=gray]  (330:3) node [bnode] {} -- (338:4) node [bnode] {} --(325:4)node [bnode] {}--cycle;
\draw (340.5:4.65) node  [color=black] {$\cdots$};
\path[fill=gray]  (342.5:5.2) node [bnode] {} -- (348:6) node [bnode] {} --(335:6)node [bnode] {}--cycle;
%\draw (0, -4.5) node [color=black] { $H_{01}^{(r)}$};
\end{tikzpicture}
\end{center}

\begin{center}
\begin{tikzpicture}[thick, scale=0.5, bnode/.style={circle, draw,
    fill=black!50, inner sep=0pt, minimum width=4pt}, enode/.style={red}]

\path[fill=gray]  (1.5,0.5)node [bnode] {}  -- (1,1) node [bnode] {} --(1.5,1.5)node
[bnode] {}-- (2,1) node [bnode]{} --cycle;
\path[fill=gray]  (1.5,-0.5)node [bnode] {}  -- (1,-1) node [bnode] {} --(1.5,-1.5)node [bnode] {}-- (2,-1) node [bnode]{} --cycle;

\foreach \x in {-2,-1,0,1,2,4,6,8,10,12,14,16,18,19,21}
    {
    \path[fill=gray]  (\x,0) node [bnode] {} -- (\x+0.5,0.5) node [bnode] {} --(\x+1,0)node [bnode] {}-- (\x+0.5,-0.5) node [bnode]{} --cycle;

}
\draw (2,0) node  [bnode] {};
%\draw (-0.5,0) node [color=black] {$\cdots$};
\draw (3.5,0) node [color=black] {$\cdots$};
\draw (5.5,0) node [color=black] {$\cdots$};
\draw (7.5,0) node [color=black] {$\cdots$};
\draw (9.5,0) node [color=black] {$\cdots$};
\draw (11.5,0) node [color=black] {$\cdots$};
\draw (13.5,0) node [color=black] {$\cdots$};
\draw (15.5,0) node [color=black] {$\cdots$};
\draw (17.5,0) node [color=black] {$\cdots$};
\draw (20.5,0) node [color=black] {$\cdots$};
\foreach \x in {6.5,10.5,14.5}
{\path[fill=gray]  (\x,2.5)node [bnode] {}  -- (\x-0.5,3) node [bnode] {} --(\x,3.5)node
[bnode] {}-- (\x+0.5,3) node [bnode]{} --cycle;

\path[fill=gray]  (\x,0.5)node [bnode] {}  -- (\x-0.5,1) node [bnode] {} --(\x,1.5)node
[bnode] {}-- (\x+0.5,1) node [bnode]{} --cycle;
\draw (\x,2.2) node [color=black] {$\vdots$};
\draw (\x-0.7,1.35) node [enode] {$\uparrow$};
}
\draw (5.25,1.7) node [enode] {$m_{1}$};
\draw (9.25,1.7) node [enode] {$m_{2}$};
\draw (13.25,1.7) node [enode] {$m_{i}$};
\foreach \x in {-2,0,1,2,4}
%\draw (0.35,0.75) node [enode] {$\leftarrow$};
%\draw (0.5,1.2) node [enode] {$t$};
\draw (3.5,0.35) node [enode] {$k_{0}$};
\draw (5.5,0.35) node [enode] {$k_{1}$};
\draw (7.5,0.35) node [enode] {$k_{2}$};
\draw (9.5,0.35) node [enode] {$k_{3}$};
\draw (11.5,0.35) node [enode] {$k_{4}$};
\draw (13.5,0.35) node [enode] {$k_{j}$};
\draw (15.5,0.35) node [enode] {$k_{j+1}$};
\draw (17.5,0.35) node [enode] {$n_{1}$};
\draw (20.5,0.35) node [enode] {$n_{2}$};
\draw (5,0) node [bnode][color=red] {};
\draw (9,0) node [bnode][color=red] {};
\draw (13,0) node [bnode][color=red] {};
\draw (17,0) node [bnode][color=red] {};
\draw (10,0) node [bnode] {};
%\draw (10, -1.5) node [color=black] { $H_{02}^{(r)}$};
\path[fill=gray]  (19,0)  -- (18.7,-0.86) node [bnode] {} -- (19,-1.3) node [bnode]{}--(19.3,-0.86)node [bnode]{} --cycle;
\end{tikzpicture}
\end{center}

\end{document}